\documentclass{article}
\usepackage{amsmath}
\usepackage{amsthm}
\usepackage{amssymb}
\usepackage{amsfonts}
\usepackage{graphicx}
\usepackage{mathtools}
\usepackage{verbatim}
\usepackage{pgf}
\usepackage{stmaryrd}
\usepackage{amscd}
\usepackage[mathscr]{euscript}
\usepackage{graphicx}
\usepackage{tikz-cd}
\usepackage{tikz}
\usepackage{scalerel}
\usetikzlibrary{decorations.pathmorphing}
\usepackage{xfrac}
\usepackage[utf8]{inputenc}
\usepackage[verbose]{wrapfig}
\usepackage{rotating}
\usepackage{geometry} 
\usepackage{fancyhdr}
\usepackage[bf,small]{titlesec}
\usepackage{tocloft}
\usepackage{multirow}
\usepackage{enumitem}
\usepackage{tocbibind}
\usepackage{cancel} 
\usepackage{bbm}
\usepackage{todonotes}
\usepackage{microtype}
\usepackage{mathpazo}
\usepackage[LGR,T1]{fontenc}
\usepackage[greek,english]{babel}
\usepackage{textalpha}
\usepackage{csquotes}
\usepackage[style=alphabetic,
			backend=biber,
			url=false,
			isbn=false,
			doi=false,
			maxnames=30]{biblatex}
\usepackage[colorlinks,bookmarksnumbered]{hyperref}

\begingroup
		\makeatletter
		\@for\theoremstyle:=definition,remark,plain\do{%
				\expandafter\g@addto@macro\csname th@\theoremstyle\endcsname{%
						\addtolength\thm@preskip\parskip
						}%
				}
\endgroup

\newcommand{\R}{{\mathbb{R}}}
\newcommand{\C}{{\mathbb{C}}}
\newcommand{\Z}{{\mathbb{Z}}}

\newcommand{\sphere}{{\mathbb{S}}}
\newcommand{\E}{{\mathbb{E}}}

\newcommand{\cat}{{\mathcal{C}}}

\newcommand{\Spc}{\mathrm{Spc}}

\newcommand{\Alg}{\mathrm{Alg}}

\newcommand{\Span}{\mathrm{Span}}

\newcommand{\id}{\text{id}}
\newcommand{\ev}{\mathrm{ev}}
\newcommand{\op}{\text{op}}

\newcommand{\cofib}{\mathrm{cofib}}

\newcommand{\surj}{\twoheadrightarrow}

\newcommand{\Spectra}{\mathrm{Sp}}

\newcommand{\Fun}{\mathrm{Fun}}
\newcommand{\Fin}{\mathrm{Fin}}

\newcommand{\Perf}{\mathscr{P}\mathrm{erf}}
\newcommand{\Cat}{\mathscr{C}\mathrm{at}}

\newcommand{\Hom}{\text{Hom}}
\newcommand{\Map}{\mathrm{Map}}

\DeclareMathOperator*{\fiberproduct}{\times}

\newcommand{\Mack}{\mathrm{Mack}}
\newcommand{\Ar}{\mathrm{Ar}}
\newcommand{\Com}{\mathrm{Com}}
\newcommand{\Env}{\mathrm{Env}}

\newcommand{\fib}{\mathrm{fib}}

\newcommand{\tr}{\mathrm{tr}}

\newcommand{\Borel}{\mathrm{Borel}}

\newcommand{\inj}{\hookrightarrow}

\DeclareMathOperator{\Mod}{Mod}

\numberwithin{equation}{section}

\theoremstyle{plain} \newtheorem{theorem}[equation]{Theorem}
\theoremstyle{plain} \newtheorem*{theorem*}{Theorem}
\theoremstyle{definition} \newtheorem{defn}[equation]{Definition}
\theoremstyle{plain} \newtheorem{prop}[equation]{Proposition}
\theoremstyle{plain} \newtheorem*{prop*}{Proposition}
\theoremstyle{plain} \newtheorem{lemma}[equation]{Lemma}
\theoremstyle{plain} \newtheorem{cor}[equation]{Corollary}
\theoremstyle{definition} \newtheorem{ex}[equation]{Example}
\theoremstyle{definition} \newtheorem{exs}[equation]{Examples}
\theoremstyle{definition} 
\theoremstyle{definition} \newtheorem{rmk}[equation]{Remark}
\theoremstyle{remark} 
\theoremstyle{definition} 
\theoremstyle{definition} \newtheorem{obs}[equation]{Observation}
\theoremstyle{plain} 
\theoremstyle{remark} 
\theoremstyle{definition} 
\theoremstyle{definition} 
\theoremstyle{definition} \newtheorem{ntn}[equation]{Notation}
\theoremstyle{plain} 
\theoremstyle{remark} 
\theoremstyle{definition} 
\theoremstyle{definition} \newtheorem{construction}[equation]{Construction}
\theoremstyle{definition} \newtheorem{variant}[equation]{Variant}
\theoremstyle{remark} \newtheorem{warning}[equation]{Warning}
\theoremstyle{definition} \newtheorem{recollection}[equation]{Recollection}

\titleformat{\subsubsection}[runin]{\bfseries}{\thesubsubsection}{1em}{}

\newcommand{\morphism}[3]{\mathop{\rotatebox{270}{$\xrightarrow{\rotatebox{90}{$\scriptstyle#2$}}$}}\limits^{#1}_{#3}}

\DeclareFontFamily{U}{cbgreek}{}
\DeclareFontShape{U}{cbgreek}{m}{n}{
        <-6>    grmn0500
        <6-7>   grmn0600
        <7-8>   grmn0700
        <8-9>   grmn0800
        <9-10>  grmn0900
        <10-12> grmn1000
        <12-17> grmn1200
        <17->   grmn1728
      }{}

\DeclareRobustCommand{\Qoppa}{%
  \text{\usefont{U}{cbgreek}{m}{n}\symbol{21}}%
}

\definecolor{DarkGreen}{RGB}{0,100,0}
\hypersetup{%
  linkcolor=DarkGreen,%
  citecolor=DarkGreen,%
  filecolor=DarkGreen,%
  menucolor=DarkGreen,%
  urlcolor=DarkGreen
}

\title{\texorpdfstring{\vspace{-2em}}{}On normed $ \E_\infty $-rings in genuine equivariant $ C_p $-spectra\texorpdfstring{\vspace{-0.5em}}{}}
\author{Lucy Yang}
\date{August 30, 2023}

\begin{document}
\maketitle
\begin{abstract}
	Genuine equivariant homotopy theory is equipped with a multitude of coherently commutative multiplication structures generalizing the classical notion of an $ \E_\infty $-algebra. 
	In this paper we study the $ C_p $-$ \E_\infty $-algebras of Nardin--Shah with respect to a cyclic group $ C_p $ of prime power order. 
	We show that many of the higher coherences inherent to the definition of parametrized algebras collapse; in particular, they may be described more simply and conceptually in terms of ordinary $ \E_\infty $-algebras as a diagram category which we call \emph{normed algebras}. 
	Our main result provides a relatively straightforward criterion for identifying $ C_p $-$ \E_\infty $-algebra structures. 
	We visit some applications of our result to real motivic invariants.   
\end{abstract}
\tableofcontents

\section{Introduction}
\subsection{Motivation}\label{subsection:motivation}
Algebraic invariants such as integral cohomology $ H^*(-;\Z) $ detect information about spaces; identifying and applying such tools form the basic premise of algebraic topology. 
Moreover, considering more structured algebraic objects leads to more refined invariants:  
Cochains with integral coefficients $ X \mapsto C^*(X;\Z) $ considered as a functor from spaces to $ \E_\infty $-$ \Z $-algebras is \emph{better} at detecting information about spaces than integral cohomology. 
For instance, $ C^*(X;\Z) $ inherits power operations as a consequence of its $ \E_\infty $-structure, and two (non-equivalent) spaces $ X $ and $ Y $ may have isomorphic cohomology rings but different power operations. 
In this way, we see that the study of structured multiplications and their operations is foundational to homotopy theory. 

In parallel, the study of structured multiplications is essential to the study of genuine equivariant homotopy types.  
Our line of inquiry naturally leads to algebraic structures whose operations are inherently genuine equivariant. 
To motivate the particular equivariant multiplicative structure we will focus on, recall that an ordinary $ \E_\infty $-algebra in spectra may be modeled by a functor satisfying certain conditions defined on  
the category of finite pointed sets \cite{segal1974categories}. 
In particular, the smash product $ A^{\otimes \ell} $ parametrizes formal sums of $ \ell $-tuples in $ A $, and said functor takes the collapse map $ \langle 2 \rangle \to \langle 1 \rangle $ to a morphism $ A^{\otimes 2} \to A $. 
In particular, $ \E_\infty $-algebra structures are governed by the category of finite pointed sets. 

In genuine equivariant homotopy theory with respect to the finite cyclic group $ C_p $ of order a prime $ p $, the role of finite sets is supplanted by \emph{finite pointed sets with $ C_p $-action} \cites[\S4]{MR3290092}[\S3.3]{hill2016equivariant}. 
The Hill--Hopkins--Ravenel norm $ N^{C_p}_e A =: A^{\otimes C_p} $ parametrizes $ | C_p| $-tuples in $ A $ indexed by a free $ C_p $-set. 
In \cite{MR3406512}, Blumberg and Hill introduced a genuine equivariant operads encoding multiplications indexed by $ G $-sets; the algebras they give rise to are called $ N_\infty $-algebras. 
In \cite[Definition 2.2.2]{NS22}, Nardin and Shah defined an $ \infty $-categorical analogue of the $ N_\infty $-algebras of Blumberg--Hill\footnote{These notions are expected to agree.}; we shall refer to the latter as $ C_p $-$ \E_\infty $-algebras (Definition \ref{defn:param_alg}, Example \ref{ex:param_calg}).  
Their structure of operations is governed by the category of finite pointed $ C_p $-sets $ \Fin_{C_p,*} $. 

Unravelling definitions, a $ C_p $-$ \E_\infty $-algebra is the data of
\begin{enumerate}[label=(\arabic*)]
	\item An underlying $ C_p $-genuine equivariant spectrum $ R $. 
	\item For each morphism of finite $ C_p $-sets $ S \to T $, a morphism of $ C_p $-spectra $ N_e^{S}R \to N_e^T R $. 
	In particular, the collapse map $ C_p \surj C_p/C_p $ indexes a morphism of $ \E_\infty $ rings $ n_R \colon N_e^{C_p}R \to R $ called the \emph{norm}. 
	\item higher coherences... 
\end{enumerate}
To exhibit a $ C_p $-$ \E_\infty $-algebra structure on a genuine $ C_p $-spectrum $ R $ is no small task. 
In the literature, one often resorts to simplifying assumptions such as requiring $ R $ to be Borel, e.g. \cite[Proposition 3.3.6]{HilmanThesis}. 
We set out to provide a relatively straightforward criterion for identifying $ C_p $-$ \E_\infty $-algebra structures. 

When $ p = 2 $, by \cite[Definition 5.2]{QScyc_equiv} the category of $ C_2 $-$ \E_\infty $-algebras is the natural domain of definition for real (i.e. $ C_2 $-equivariant) topological Hochschild homology and other real motivic invariants. 
This work grew out of the author's interest in real motivic invariants and will be used in upcoming work on a real version of the Hochschild--Kostant--Rosenberg theorem. 

\subsection{Main result}
To motivate our main result, note that any morphism of $ C_p $-sets $ S \to C_p/C_p $ can be expressed as the composite of `collapse the free $ C_p $-orbits' followed by a (non-equivariant) map of finite sets.  
Thus a $ C_p $-$ \E_\infty $ algebra, regarded as functor defined on $ C_p $-sets, determines two pieces of data: its restriction to sets on which $ C_p $ acts trivially, and its value on collapse maps. 
The former specifies an $ \E_\infty $-algebra structure, while the latter specifies a norm map $ n $. 
To assert that these data are \emph{enough} to specify a $ C_p $-$ \E_\infty $-algebra structure means that any higher coherence conditions on the norm map $ n $ collapse. 
We might hope that this is indeed the case, since the category $ \Fin_{C_p,*} $ is \emph{freely} generated by the $ C_p $-set $ C_p/C_p $. 

Let $ A \in \E_\infty\Alg(\Spectra^{BC_p} ) $ be an $ \E_\infty $-algebra with naïve $ C_p $-action and $ \sigma \in C_p $ a generator. 
\begin{obs} [Observation \ref{obs:clarifying_actions}]
	Write $ A^{\otimes^\Delta p} $ for the object in $ \E_\infty\Alg(\Spectra^{BC_p} ) $ with the \emph{diagonal} action, i.e. such that $ \sigma $ acts by $ \sigma (a_1 \otimes \cdots \otimes a_p) = \sigma(a_p) \otimes \sigma (a_1) \otimes \cdots \otimes \sigma(a_{p-1}) $. 
	Write $ A^{\otimes^\tau p} $ for the object in $ \E_\infty\Alg(\Spectra^{BC_p} ) $ with the \emph{transposition} action, i.e. such that $ \sigma $ acts by $ \sigma (a_1 \otimes \cdots \otimes a_p) = a_p \otimes a_1 \otimes \cdots \otimes a_{p-1} $.  
	Then the endomorphism $ \id_A \otimes \sigma \otimes \cdots \otimes \sigma^{p-1} $ of $ A^{\otimes p} \in \E_\infty \Alg(\Spectra) $ promotes to an equivalence $ A^{\otimes^\Delta p} \to A^{ \otimes^\tau p} $ in $ \E_\infty \Alg(\Spectra)^{BC_p} $--in particular it is $ C_p $-equivariant.
\end{obs}
\begin{defn} [Definition \ref{defn:normedrings}]
	Write $ \mathcal{O}_{C_p} $ for the category of finite sets with transitive $ C_p $-action, and let $ \sigma \in C_p $ be a generator.
	A \emph{normed $ \E_\infty $-algebra} in $ C_p $-spectra is the data of an $ \E_\infty $-algebra $ A $ in $ \Spectra^{C_p} $, a morphism of $ \E_\infty $-rings $ n_A \colon N^{C_p}(A^e_{hC_p}) \to A $, and a homotopy making the following diagram $ \mathcal{O}_{C_p} \to \E_\infty \Alg(\Spectra)^{BC_p}  $
	\begin{equation*}
		\begin{tikzcd}[row sep=tiny]
		  (A^e)^{\otimes^\tau p} \ar[dd,"\rotatebox{90}{$\sim$}","{\id \otimes \sigma \otimes \cdots \otimes \sigma^{p-1}}"'] \ar[rd,"{n_A^e}", bend left=10]& \\
		& A^e \\
		 (A^e)^{\otimes^\Delta p} \ar[ru,"{m_A}"', bend right=10] &
		\end{tikzcd}
	\end{equation*}
	commute, where the $ C_p $-action on $ (A^e)^{\otimes^\Delta p} $ corresponding to the inclusion $ BC_p \subseteq \mathcal{O}_{C_p} $ is the transposition.
\end{defn}

The main result of this paper both formalizes and confirms the aforementioned intuitive picture. 

\begin{theorem}  [Corollary \ref{cor:operadic_diagrammatic_comparison} \& Theorem \ref{thm:mainthm}] 
	The canonical forgetful functor from the category of $ C_p $-$\E_\infty $ algebras in $ C_p $-spectra (Example \ref{ex:param_calg}) to the category of normed $ \E_\infty $-algebras in $ C_p $-spectra is an equivalence.  
\end{theorem}
A key input to the proof of Theorem \ref{thm:mainthm} is an explicit description of the free $ C_p $-$ \E_\infty $ algebra in $ \Spectra^{C_p} $ on an $ \E_\infty $-algebra $ A $ in $ \Spectra^{C_p} $.
By Theorem \ref{theorem:freenormedalg_formula} and Proposition \ref{prop:eqvtspectrarecollement}, the underlying $ C_p $-spectrum of the free $ C_p $-$ \E_\infty $ algebra $ F(A) $ on $ A $ is given by 
\begin{equation*}
F(A) \simeq
\begin{tikzcd}[cramped]
	& A^{\varphi{C_p}}  \otimes A^e_{hC_p} \ar[d,"{s_A \otimes \nu_A}"] \\
	A^e \ar[r,mapsto] & A^{tC_p}
\end{tikzcd}
\end{equation*}
where $ u $ is the unit, $ s_A : A^{\varphi{C_p}} \to A^{tC_p} $ is the structure map, and $ \nu_A $ is the twisted Tate-valued norm (Definition \ref{defn:twisted_tate_frob}). 

\begin{rmk}
	There is an analogous statement (Theorem \ref{thm:relative_mainthm}), proved in essentially the same way, for relative normed algebras, i.e. $ C_p $-$ \E_\infty $-algebras \emph{over} a fixed $ C_p $-$ \E_\infty $-algebra $ A $. 
\end{rmk}
One expects an analogous result to hold for arbitrary $ G $, but we stick with $ C_p $ here because the author's motivating example is the case $ G = C_2 $, and because the complexity of (\ref{diagram:normedrings}) seemingly grows exponentially in the subgroup lattice of $ G $. 

\subsection{Applications \& Examples}
The power of Theorem \ref{thm:mainthm} is that, in many cases, it is easier to identify objects in the diagram category Definition \ref{defn:normedrings} than to produce a $ C_p $-$ \E_\infty $-algebra in the sense of Definition \ref{defn:param_alg}, which requires exhibiting an infinite amount of coherence data. 
In particular, a normed ring is the data of an $ \E_\infty $-ring in $ \Spectra^{C_p} $ plus the additional datum of a commutative diagram (\ref{eq:tatefrobenius_lift}). 
As an application, in \S\ref{section:app_ex} we show that various $ \E_\infty $-rings in $ \Spectra^{C_p} $ admit natural lifts to $ C_p $-$ \E_\infty $-rings in $ \Spectra^{C_p} $. 
\begin{cor} [Theorem \ref{thm:constMackey_is_normed}]
	Let $ k $ be a discrete commutative ring. 
	The constant $ C_p $-Mackey functor $ \underline{k} $ on $ k $ acquires an essentially unique structure of a $ C_p $-$ \E_\infty $-ring. 	
\end{cor}
Using our main theorem, we are able to give an alternative proof of a special case of a result \cite[Proposition 3.3.6]{HilmanThesis} of Kaif Hilman. 
In view of the expected correspondence between $ N_\infty $-algebras and $ C_p $-$ \E_\infty $-algebras, the following result should also be compared to Theorem 6.26 of \cite{MR3406512}. 
\begin{cor}
	[Proposition \ref{prop:Borel_are_normed}] 	 
	Every Borel $ \E_\infty$-algebra in $ C_p $-spectra admits an essentially unique refinement to a $ C_p $-$ \E_\infty $-algebra. 
\end{cor}
Many examples arise in the case $ p = 2 $ because involutions are ubiquitous in topology. 
Natural examples of $ \E_\infty $-rings in $ \Spectra^{C_2} $ include real topological and algebraic K-theories (\cite{MR206940} \& \ref{defn:real_alg_Ktheory}). 
\begin{cor} [Corollary \ref{cor:realKtheory_normed}]
\begin{itemize}
	\item Real topological K-theory $ KU_{\R} $ admits a unique refinement to a $ C_2 $-$ \E_\infty $ ring spectrum. 
	\item If $ A $ satisfies the \emph{homotopy limit problem}, then $ K_{\R}(A) $ admits a unique refinement to a $ C_2 $-$\E_\infty $ ring spectrum. 
\end{itemize}	
\end{cor}
A slightly less trivial class of examples are provided by the following
\begin{cor} [Proposition \ref{prop:norms_are_normed}]
	Let $ B \in \E_\infty\Alg (\Spectra) $ be an $ \E_\infty $-algebra. 
	Then $ N^{C_p}B $ admits a canonical structure of a $ C_p $-$ \E_\infty $-algebra. 
\end{cor}
\begin{rmk}
	Real motivic invariants and their associated real trace theories provided the impetus for this work. 
	In particular, Theorem \ref{thm:constMackey_is_normed} will be used in the author's upcoming work on real trace theories.
\end{rmk}

\subsection{Outline}
Despite the nice intuitive picture outlined in \S\ref{subsection:motivation}, handling the higher coherence conditions associated to $ n $ gets complicated quickly. 
Thus our proof strategy does not appeal directly to understanding the operadic indexing category (although we will need some understanding of this to write down a comparison functor). 

In \S\ref{section:background}, we collect background on genuine equivariant homotopy theory, as well as the parametrized $ \infty $-categorical perspective on equivariant algebras. 
In \S\ref{section:normedrings}, we define normed rings. 
In \S\ref{subsection:comparison_functor}, we define a comparison functor from parametrized algebras to normed rings. 
In \S\ref{subsection:param_monoidal_envelope}, we exhibit a formula for the free $ C_p $-$ \E_\infty $-algebra on an $ \E_\infty $-algebra. 
In \S\ref{subsection:mainthm_proof}, we show that the free $ C_p $-$ \E_\infty $-algebra on an $ \E_\infty $-algebra also computes the free \emph{normed algebra}, and conclude by the Barr--Beck--Lurie theorem. 
In \S\ref{section:app_ex}, we look into a few examples and applications. 

\subsection{Background, Notation, \& Conventions} 
We assume some familiarity with the language of $ \infty $-categories (in the form of quasi-categories) as introduced by Joyal \cite{joyal-qc} and developed in \cite{LurHTT}. 
All categories are understood to be $ \infty $-categories unless otherwise specified. 
We do a cursory review of the theory of parametrized $ \infty $-categories as developed by Barwick, Dotto, Glasman, Nardin, and Shah \cite{BDGNSintro,BDGNS1,BDGNS9,Nardinthesis,Shah18}, but the reader should consult the former references for more details. 
We will assume some familiarity with the $ \infty $-operads of \cite[Chapters 2 \& 3]{LurHA}, which we will compare to the parametrized algebras of \cite{NS22}. 

To reduce visual clutter, we regularly drop subscripts such as a prime $ p $ or a ($ C_p $-)$ \E_\infty $-algebra $ A $ when they are understood to be fixed (e.g. within the proof of a particular proposition). 

\subsection{Acknowledgements} The author is indebted to Denis Nardin and Jay Shah for numerous enlightening conversations and for, in collaboration Clark Barwick, Emanuele Dotto, and Saul Glasman, setting up the foundations of parametrized $ \infty $-categories.
The author would like to thank Elden Elmanto and Noah Riggenbach for helpful conversations, and Andrew Blumberg and Elden Elmanto for feedback on an early draft. 
The author was supported by a NSF Graduate Research Fellowship under Grant No. DGE 2140743 during the completion of this work. 

\section{Background}\label{section:background}
We collect some background on genuine equivariant homotopy theory and parametrized $ \infty $-categories here. 
In \S\ref{subsection:param} and \S\ref{subsection:param_alg}, we recall the parametrized $ \infty $-categorical language and parametrized algebras, resp. of Barwick--Dotto--Glasman--Nardin--Shah. 
In \S\ref{subsection:background_geneqvt}, we collect background and structural results on the $ C_p $-genuine equivariant category. 

\subsection{Parametrized \texorpdfstring{$ \infty $}{∞}-categories}\label{subsection:param}
Let $ G $ be a finite group.  
\begin{recollection}\label{rec:orbit_cat}
	The orbit category $ \mathcal{O}_G $ is the category with objects finite transitive $ G $-sets and morphisms $ G $-equivariant maps. 
	We let $ \Fin_G $ denote the finite coproduct completion of $ \mathcal{O}_G $, i.e. the category of finite $ G $-sets and $ G $-equivariant maps. 
	We recall that $ \mathcal{O}^\op_G $ is an \emph{orbital} $ \infty $-category in the sense of Definition 1.2 of \cite{Nardinthesis}. 
\end{recollection}
\begin{defn}
	(\cites[between Examples 1.3 \& 1.4]{Nardinthesis}[Definition 1.3]{BDGNS1}) A \emph{$ G $-$ \infty $-category} is a cocartesian fibration $ \cat \to \mathcal{O}_G^\op $. 

	\cite[beginning of \S1.2]{Nardinthesis} A morphism of $ G $-$ \infty $-categories is a functor $ F $ of $ \infty $-categories over $ \mathcal{O}_G^\op $:
	\begin{equation*}
	\begin{tikzcd}[column sep=tiny,row sep=small]
		\cat \ar[rd,"p"'] \ar[rr,"F"] & & \mathcal{D} \ar[ld,"q"] \\
		& \mathcal{O}_G^\op &
	\end{tikzcd}
	\end{equation*}
	which takes $ p $-cocartesian arrows in $ \cat $ to $ q $-cocartesian arrows in $ \mathcal{D} $. 
	We denote the category of $ G $-functors by $ \Fun_{G}(\cat, \mathcal{D}) $. 
\end{defn}
\begin{rmk}
	\cites[\S3.2.2]{LurHTT}[Example 2.5]{Shah18} \label{rmk:param_unstraighten}
	Let $ \Cat_\infty $ denote the large $ \infty $-category of small $ \infty $-categories. 
	There is a universal cocartesian fibration $ \mathcal{U} \to \Cat_\infty $ such that pullback induces an equivalence
	\begin{equation*}
		\Fun(\mathcal{O}_G^\op, \Cat_\infty) \simeq \Cat_{\infty/\mathcal{O}_G^\op}^{\mathrm{cocart}} .
	\end{equation*}
	Unraveling definitions and taking $ G = C_p $, a $ C_p $-$ \infty $-category is the data of 
	\begin{itemize}
		\item an $ \infty $-category $ \cat^{C_p} $,
		\item an $ \infty $-category with $ C_p $-action $ \cat^{e} $, and 
		\item a functor $ \cat^{C_p} \to \cat^e $ which lifts along the $ C_p $ homotopy fixed points $ (\cat^e)^{hC_p} \to \cat^e $. 
		In particular, if $ \cat^e $ is endowed with the trivial $ C_p $-action, then $ (\cat^e)^{hC_p} \simeq (\cat^e)^{BC_p} \simeq \Fun(BC_p , \cat^e) $ comprises objects in $ \cat^e $ with (naïve) $ C_p $-action.
	\end{itemize}
	In particular, we see that a cocartesian section $ \sigma \colon \mathcal{O}^\op_{C_p} \to \cat $ is determined by its value on $ \sigma(C_p/C_p) $. 
	Informally, we regard the category of cocartesian sections of $ \cat $ as the category of objects in $ \cat $. 
\end{rmk}
\begin{ntn}
	Going forward, we use the notation $ \mathcal{T} $ for $ \mathcal{O}_{C_p}^\op $ to reduce notational clutter. 
	While most of the general theory in \S\ref{subsection:param} and \S\ref{subsection:param_alg} applies to $ \mathcal{T} $ a general atomic orbital $ \infty $-category, we will not need this level of generality to formulate our main results. 
\end{ntn}
There is an (internal to $ \mathcal{T} $-parametrized categories) version of functor categories. 
The notion of \emph{parametrized functor categories} of \cite[\S3]{Shah18} will be necessary to our investigation of parametrized colimits. 
\begin{prop} \cites[Proposition 3.1]{Shah18}[Construction 5.2]{BDGNS1}\label{prop:param_functors}
	Let $ \cat \to \mathcal{T}^\op $, $ \mathcal{D} \to \mathcal{T}^\op $ be cocartesian fibrations. 
	Then there exists a cocartesian fibration $ \underline{\Fun}(\cat, \mathcal{D}) \to \mathcal{T}^\op $ such that under the straightening-unstraightening equivalence of Remark \ref{rmk:param_unstraighten}, $ \underline{\Fun}(\cat, \mathcal{D}) $ represents the presheaf
	\begin{equation*}
		\mathcal{E} \mapsto \hom_{\mathcal{T}^\op}\left(\mathcal{E} \times_{\mathcal{T}^\op} \cat, \mathcal{D} \right).
	\end{equation*}
\end{prop}
Notice that an object of $ \underline{\Fun}(\cat, \mathcal{D}) $ over $ t \in \mathcal{T} $ is a $ (\mathcal{T}^\op)_{t/} $-functor
\begin{equation*}
	(\mathcal{T}^\op)_{t/} \times_{\mathcal{T}^\op} \cat \to (\mathcal{T}^\op)_{t/} \times_{\mathcal{T}^\op} \mathcal{D} .
\end{equation*}
\begin{construction}
	[$\mathcal{T}$-category of objects] \label{cons:Tcat_of_objects}
	\cite[Definition 7.4]{BDGNS1} 
	Let $ E $ be a (non-parametrized) $ \infty $-category. 
	The product $ E \times \mathcal{T}^\op $ may be regarded as a $ \mathcal{T} $-$ \infty $-category via projection onto the second factor. 
	Evaluation at the source exhibits the (non-parametrized) functor category $ \Fun(\Delta^1, \mathcal{T}^\op) \xrightarrow{\ev_0} \mathcal{T}^\op  $ as a cartesian fibration. 
	The parametrized functor category of \cite[Recollection 5.1]{BDGNS1}
	\begin{equation*}
		\underline{E}_{\mathcal{T}} := \underline{\Fun}_{\mathcal{T}^\op}\left(\Fun(\Delta^1, \mathcal{T}^\op), E \times \mathcal{T}^\op \right)
	\end{equation*}
	is the \emph{$ \mathcal{T} $-$ \infty$-category of $ \mathcal{T} $-objects in $ E $}. 
\end{construction}
\begin{theorem}\label{thm:Tcat_of_objects_cofree}
	\cite[Theorem 7.8]{BDGNS1} 
	Let $ \cat $ be a $ \mathcal{T} $-$ \infty $-category. 
	Let $ \mathcal{D} $ be an $ \infty $-category. 
	Then the $ \mathcal{T} $-category of objects of Construction \ref{cons:Tcat_of_objects} satisfies
	\begin{equation*}
		\Fun_{\mathcal{T}^\op}(\cat, \underline{\mathcal{D}}) \simeq \Fun(\cat, \mathcal{D}).
	\end{equation*}
\end{theorem}
\begin{ex}
	Taking $ E = \Spc $ and $ \cat = \mathcal{T}^\op $ in Theorem \ref{thm:Tcat_of_objects_cofree}, we see that cocartesian sections of $ \underline{\Spc}_{\mathcal{T}} $ correspond exactly to $ \Fun(\mathcal{T}^\op, \Spc) $.
\end{ex}

We will need to know what a $ G $-left Kan extension is. 
In service of keeping the background section brief, we take Remark 10.2(3) of \cite{Shah18}, which is equivalent to Definition 10.1 of \emph{loc.cit}.
\begin{ntn}\label{ntn:param_fiber}
	\cite[Notation 2.29]{Shah18} 
	Let $ p: \mathcal{D} \to \mathcal{T}^\op $ be a $ \mathcal{T} $-$\infty $-category. 
	Given an object $ x \in \mathcal{D} $, define
	\begin{equation*}
		\underline{x} := \{x\} \times_{\mathcal{D}} \Ar^{cocart}(\mathcal{D}) .
	\end{equation*}
	Given a $ \mathcal{T} $-functor $ \psi: \cat \to \mathcal{D} $, define the \emph{parametrized fiber of $ \psi $ over $ x \in \mathcal{D}$} to be
	\begin{equation*}
		\cat_{\underline{x}} := \underline{x} \fiberproduct_{\mathcal{D}, \psi} \cat .
	\end{equation*}
	Observe that $ \cat_{\underline{x}} $ may be naturally regarded as a $ \left(\mathcal{T}^\op\right)^{p(x)/-} $-category. 
\end{ntn}
\begin{defn}
	\cite[Remark 10.2(3)]{Shah18} \label{defn:param_Kan_extension}
	Suppose given a diagram of $ \mathcal{T} $-$ \infty $-categories
	\begin{equation*}
	\begin{tikzcd}
		C \ar[d,"\psi"']  \ar[r,"F"] & E \\
		D \ar[ru,"{}"{name=E}, bend right=20,"G"'] & \ar[Rightarrow,from=1-1, to=E,"\eta"]
	\end{tikzcd}.
	\end{equation*}
	We say that $ G $ is a \emph{left $ \mathcal{T} $-Kan extension of $ F $ along $ \psi $} if for all $ t \in \mathcal{T} $ and all $ x \in D_{t} $, $ \left.G\right|_{\underline{x}} $ is a left $( \mathcal{T}^{\op})^{t/} $-Kan extension of $ \left. F\right|_{\underline{x}} \colon C_{\underline{x}} \to E_{\underline{t}} $ along $ \psi_{\underline{x}} $. 
\end{defn}

\subsection{Genuine equivariant homotopy theory}\label{subsection:background_geneqvt}
In this section, we introduce the stable $ C_p $-genuine equivariant category, discuss a parametrized lift (Example \ref{ex:param_Gspectra}) and give an alternative presentation (Proposition \ref{prop:eqvtspectrarecollement}) which will be useful to our study of algebras.  
Finally, we recall the Hill--Hopkins--Ravenel norms. 
\begin{prop} \label{prop:spancatcoherence}
	Let $ G $ be a finite group. 
	Then there exists an $ \infty $-category $ \Span\left(\Fin_G\right) $ having
	\begin{itemize} 
		\item the same objects as $ \Fin_G $
		\item homotopy classes of morphisms from $ V $ to $ U $ in $ \Span(\Fin_G) $ are in bijection with diagrams $ {V \leftarrow T \to U} $ up to isomorphism of diagrams fixing $ V $ and $ U $. 
		\item The composite of $ V \leftarrow T \to U $ and $ U \leftarrow S \to W $ is equivalent to the diagram $ V \leftarrow T \times_U S \to W $.
	\end{itemize}
	Moreover, $ \Span(\Fin_G) $ is semiadditive, i.e. finite coproducts and products are isomorphic, and are given on underlying $ G $-sets by the disjoint union. 
\end{prop}
\begin{proof}
	The construction of $ \Span(\Fin_G) $ is \cite[Proposition 5.6]{BarwickMackey} applied to \cite[Example 5.4]{BarwickMackey}. 
	The (0-)semiadditivity of $ \Span(\Fin_G) $ follows from noticing that $ \Span(\Fin_G) $ is a module over $ \Span(\Fin) $ and \cite[Corollary 3.19]{MR4133704}. 
\end{proof}
The notion of a Mackey functor first appeared in \cite{MR0360771} in algebra and in \cite{MR1413302} in homotopy theory; the following $ \infty $-categorical version of the definition is taken from \cite[\S2.3]{Nardinthesis}. 
\begin{defn}  
	Let $ G $ be a finite group and let $ \Span(\Fin_G) $ be the span category of Proposition \ref{prop:spancatcoherence}. 
	Let $ \cat $ be a category which admits finite products. 
	Then the category of $ \cat $-valued \emph{$ G $-Mackey functors} is given by 
	\begin{equation*} 
		\Mack_G(\cat) := \Fun^\Sigma(\Span(\Fin_{G}), \cat) 
	\end{equation*}
	where the right-hand side denotes the full subcategory on functors which take direct sums in $ \Span(\Fin_{G}) $ to products in $ \cat $. 
	We will denote the category of \emph{genuine equivariant $ G $-spectra} by $ \Spectra^G = \Mack_G(\Spectra) $. 
\end{defn}
We identify the theory of orthogonal $ G $-spectra (where weak equivalences are detected levelwise) with $ G $-spectral Mackey functors via the equivalence established in \cite[\S3]{GuillouMay}. 
\begin{recollection} [Smash product of $ G $-Mackey functors]
	The category $ \Span(\Fin_{G}) $ inherits a symmetric monoidal structure from $ \Fin_{G} $ given on underlying objects by cartesian product of finite $ G $-sets \cite[Proposition 2.9]{BarwickGlasmanShah}. 
	Suppose that $ \cat $ has a presentably symmetric monoidal structure\footnote{That is, the tensor product commutes with (small) colimits separately in each variable.} $ \otimes $. 
	Then we can equip $ \Mack_G(\cat) = \Fun^\Sigma(\Span(\Fin_{G}), \cat) $ with a symmetric monoidal structure given by Day convolution \cite[Proposition 2.11]{MR3601070}.
	When we take $ \cat = \Spectra $ and the symmetric monoidal structure to be the smash product on spectra, this recovers the usual smash product of $ G $-spectra. 
\end{recollection}\label{rec:eqvtsmashproduct}
The $ \infty $-category of $ G $-Mackey functors in spectra is equivalent to the category of cocartesian sections of a $ G $-parametrized $ \infty $-category. 
\begin{ex}\label{ex:param_Gspectra}
	The $ G $-$ \infty $-category of $ G $-spectra $ \underline{\Spectra}^G $ is \cite[Definition 7.3 \& Corollary 7.4.1]{BDGNS4} applied to $ D = \underline{\Spc}^G $. 
\end{ex}

There is an alternative way of understanding $ \Mack_{C_p}(\cat) $ as a recollement \emph{when $ \cat $ is stable} and admits $ BC_p $-shaped colimits. 
The following is \cite[Theorem 6.24]{MNN}. 
\begin{prop}\label{prop:eqvtspectrarecollement} 
	There is an equivalence of stable $ \infty $-categories
	\begin{equation*}
	\begin{split}
		&\Spectra^{C_p} = \Mack_{C_p}(\Spectra) \to \Spectra^{BC_p} \times_{\Spectra} Ar(\Spectra) \\
		&X \mapsto \left(X^e, \cofib\left((X^e)_{hC_p} \xrightarrow{\tr} X^{C_p}\right) \to (X^e)^{tC_p} \right)
	\end{split}
	\end{equation*}  
	where the map $ Ar(\Spectra) \to \Spectra $ is evaluation at the target. 
	We call $  X^{\varphi{C_p}} := \cofib\left((X^e)_{hC_p} \xrightarrow{\tr} X^{C_p}\right) $ the \emph{$C_p$-geometric fixed points} of $ X $. 
\end{prop}
\begin{ntn}
	We will denote the projection $ \Spectra^{C_p} \to \Ar\left( \Spectra\right) $ by $ s_{(-)} $, i.e. for any $ C_p $-spectrum $ A $ we have a map $ s_A \colon A^{\varphi C_p} \to A^{tC_p} $. 
\end{ntn}
It will be convenient to know that the recollement of Proposition \ref{prop:eqvtspectrarecollement} is compatible with symmetric monoidal structures. 
\begin{prop}
	Let $ C_p $ be a cyclic group of prime power order. 
	Then the recollement of Notation \ref{prop:eqvtspectrarecollement} is a symmetric monoidal recollement in the sense of \cite[Definition 2.20]{ShahRS}. 
\end{prop}
\begin{cor} \label{cor:eqvtalg_recollement}
	Let $ C_p $ be a cyclic group of prime power order. 
	Then there is an equivalence of $ \infty $-categories
	\begin{equation*}
		\Alg_{\E_\infty} \Spectra^{C_p} \xrightarrow{\sim} \Alg_{\E_\infty} \Spectra^{BC_p} \times_{\Alg_{\E_\infty}\Spectra} Ar(\Alg_{\E_\infty}\Spectra) 
	\end{equation*}
	such that applying forgetful functors recovers the equivalence of Proposition \ref{prop:eqvtspectrarecollement}.
\end{cor}
\begin{proof}
	The corollary follows from \cite[Theorem 1.2]{ShahRS} and the definition of $ \E_\infty $-algebras. 
\end{proof}
\begin{obs}
	Now suppose $ A, B \in \E_\infty \Alg\left(\Spectra^{C_p}\right) $. 
	Then the morphism space is computed as
	\begin{equation*}
		\hom_{\E_\infty \Alg(\Spectra^{C_p})} (A,B) \simeq \hom_{\E_\infty \Alg(\Spectra^{BC_p})} (A^e,B^e) \fiberproduct_{\hom_{\E_\infty \Alg(\Spectra)} \left(A^{tC_p},B^{tC_p}\right)} \hom_{\Ar(\E_\infty\Alg(\Spectra))}
		\left(\begin{tikzcd}[row sep=small] A^{\varphi C_p} \ar[d] \\ A^{tC_p}\end{tikzcd},\begin{tikzcd}[row sep=small]B^{\varphi C_p} \ar[d]\\ B^{tC_p}\end{tikzcd}\right)
	\end{equation*}
\end{obs}

\begin{recollection}
	[Tate diagonal] \cite[Definition III.1.4]{nikolaus-scholze}\label{rec:tate_diagonal}
	The Tate diagonal is a natural transformation $ \id \to (-^{\otimes p})^{tC_p} $ of exact functors $ \Spectra \to \Spectra $ where $ C_p $ acts on $ (A)^{\otimes p} $ via a cyclic permutation. 
\end{recollection}
\begin{recollection}\label{rec:HHRnorm}
	Given a subgroup inclusion $ H \subset G $, the Hill--Hopkins--Ravenel norm \cite[Definition A.52]{MR3505179} is a (non-exact) functor
	\begin{equation*}
		N_H^G: \Spectra^H \to \Spectra^G .
	\end{equation*}
	When $ H = \{e\} \subseteq G = C_p $, the norm is uniquely characterized by the existence of natural equivalences $ \left(N_e^{C_p}X\right)^{\varphi C_p} \simeq X $ in $ \Spectra^{\{e\}} \simeq \Spectra $ and $ \left(N_e^{C_p}X \right)^e \simeq X^{\otimes p} $ in $ \Spectra^{BC_p} $, where $ C_p $ acts on the smash product by permuting the terms. 
	The connecting map $ X \to ( X^{\otimes p})^{tC_p} $ is given by the Tate diagonal of \cite[Theorem 1.7]{nikolaus-scholze}. 
	The functor $ N_H^G $ enjoys the properties of being symmetric monoidal and it preserves sifted colimits \cite[Proposition A.54]{MR3505179}, so it lifts to a functor \cite[Proposition A.56]{MR3505179}
	\begin{equation*}
		N_H^G: \Alg_{\E_\infty} \Spectra^H \to \Alg_{\E_\infty} \Spectra^G .
	\end{equation*}
\end{recollection}
\begin{lemma}\label{lemma:norm_cocontinuous_alg}
	The Hill--Hopkins--Ravenel norm $ N^{C_p}: \E_\infty\Alg (\Spectra) \to \E_\infty\Alg\left(\Spectra^{C_p}\right) $ preserves all small colimits.  
\end{lemma}
\begin{proof}
	By \cite[Lemma 2.8]{norms}, it suffices to show that $ N^{C_p} $ preserves sifted colimits and finite coproducts. 
	The norm $ N^{C_p} $ preserves sifted colimits of algebras because they are computed at the level of underlying spectra, and $ N^{C_p} $ preserves finite coproducts of algebras because it is symmetric monoidal with respect to the smash product on $ \Spectra $ and $ \Spectra^{C_p} $ (Recollections \ref{rec:eqvtsmashproduct} \& \ref{rec:HHRnorm}). 
\end{proof}

\subsection{\texorpdfstring{$ C_p $-$ \E_\infty $}{Cp-E∞}-rings}\label{subsection:param_alg}
In this section we introduce the genuine equivariant algebraic structures of interest via the formalism of parametrized operads of Nardin--Shah \cite{NS22}. 
We fix notation for the remainder of the paper. 
\begin{ntn}
	Let $ p $ be a prime and let $ \mathcal{T} $ denote the orbital $ \infty $-category $ \mathcal{O}_{C_p} $ of Recollection \ref{rec:orbit_cat}. 
\end{ntn}
\begin{defn}
	The category $ \underline{\Fin}_{\mathcal{T}} = \underline{\Fin}_{C_p} $ of \emph{parametrized $ \mathcal{T} $-sets} is the $ \mathcal{T} $-$ \infty $-category classified by the functor $ V \mapsto \Fin_{\mathcal{T}^{/V}} $.
	Equivalently, it is described by the fiber product $ \Ar\left(\Fin_{\mathcal{T}^{/V}}\right) \fiberproduct_{\Fin_{\mathcal{T}}} \{V\} $. 
	The category $ \underline{\Fin}_{\mathcal{T},*} $ of \emph{parametrized pointed $ \mathcal{T} $-sets} is the $ \mathcal{T} $-$ \infty $-category classified by the functor $ V \mapsto \left(\Fin_{\mathcal{T}^{/V}}\right)_{\id_V/} $. 
	\begin{equation*}
		\underline{\Fin}_{\mathcal{T}}^v := \Ar\left(\Fin_{\mathcal{T}^{/V}}\right) \fiberproduct_{\Fin_{\mathcal{T}}} \mathcal{T}
	\end{equation*}
\end{defn}
\begin{ex}
	We unpack the definition in the case $ \mathcal{T} = \mathcal{O}_{C_p} $. 
	For each orbit, the fiber is given by
	\begin{equation*}
		\left(\underline{\Fin}_{\mathcal{T}}^v\right)_{C_p/C_p} \simeq \Fin_{C_p} \qquad	\left(\underline{\Fin}_{\mathcal{T}}^v\right)_{C_p} \simeq \Fin_{C_p}^{Free} 		
	\end{equation*} 
	and the morphism $ C_p/C_p \leftarrow C_p $ classifies the functor $ \Fin_{C_p} \to \Fin_{C_p}^{Free} $, $ V \mapsto V \fiberproduct_{C_p/C_p} C_p $. 
\end{ex}
\begin{defn}
	\cite[Definition 2.1.2]{NS22} 
	Let $ \mathcal{T} $ be an atomic orbital $ \infty $-category. 
	The ($ \mathcal{T} $-parametrized) $ \infty $-category of \emph{finite pointed $ \mathcal{T} $-sets} is
	\begin{equation*}
		\underline{\Fin}_{\mathcal{T},*} = \Span\left(\underline{\Fin}_{\mathcal{T}}^v,\left( \underline{\Fin}_{\mathcal{T}}^v \right)^{si}, \left( \underline{\Fin}_{\mathcal{T}}^v \right)^{tdeg} \right). 
	\end{equation*}
	where a morphism $ [\phi: f \to g] $ of $ \underline{\Fin}_{\mathcal{T}}^v $ 
	\begin{equation*}
		\begin{tikzcd}[column sep=small, row sep=small]
			U \ar[r,"h"] \ar[d,"f"'] & X \ar[d,"g"] \\
			V \ar[r,"k"] & Y
		\end{tikzcd}
	\end{equation*}
	\begin{itemize}
		\item belongs to $ \left( \underline{\Fin}_{\mathcal{T}}^v \right)^{tdeg} $ if $ k $ is degenerate, and
		\item belongs to $ \left( \underline{\Fin}_{\mathcal{T}}^v \right)^{si} $ if $ U \to V \times_Y X $ is a summand inclusion. 
	\end{itemize}
\end{defn}
\begin{defn}\label{defn:param_operad}
	\cite[Definition 2.1.7]{NS22}
	A \emph{$ \mathcal{T} $-$ \infty $-operad} is a pair $ (\cat^\otimes,p) $ consisting of a $ \mathcal{T} $-$\infty $-category $ \cat^\otimes $ and a $ \mathcal{T} $-functor $ p: \cat^\otimes \to \underline{\Fin}_{\mathcal{T},*} $ which is a categorical fibration and satisfies the following additional conditions
	\begin{enumerate}[label=(\arabic*)]
		\item For every inert morphism $ \psi: f_+ \to g_+ $ of $ \underline{\Fin}_{\mathcal{T},*} $ and every object $ x \in \cat^\otimes_{f_+} $, there is a $ p $-cocartesian edge $ x \to y $ covering $ \psi $.
		\item \label{defitem:param_segal_condition} For any object $ f_+ = [U_+ \to V ] $ of $ \underline{\Fin}_{\mathcal{T},*} $, the $ p $-cocartesian edges lying over the characteristic morphisms
		\begin{equation*}
			\left\{\chi_{[W \subseteq U]} : f_+ \to I(W)_+ \mid W \in \mathrm{Orbit}(U) \right\}
		\end{equation*}
		together induce an equivalence
		\begin{equation*}
			\prod_{W \in \mathrm{Orbit}(U)} \left(\chi_{[W \subseteq U]}  \right)_! \colon \cat^\otimes_{f_+} \xrightarrow{\sim}  \prod_{W \in \mathrm{Orbit}(U)} \cat^\otimes_{I(W)_+}.
		\end{equation*}
		\item \label{defitem:param_tensor_pushforward} For any morphism
		\begin{equation*}
			\psi \colon f_+ = [U_+ \to V] \to g_+ = [U'_+ \to V']
		\end{equation*}
		of $ \underline{\Fin}_{\mathcal{T},*} $, objects $ x \in \cat^\otimes_{f_+} $ and $ y \in \cat^\otimes_{g_+} $, and any choice of $ p $-cocartesian edges
		\begin{equation*}
			\{y \to y_W \mid W \in \mathrm{Orbit}(U') \}
		\end{equation*}
		lying over the characteristic morphisms
		\begin{equation*}
			\left\{\chi_{[W \subseteq U]} : g_+ \to I(W)_+ \mid W \in \mathrm{Orbit}(U') \right\},
		\end{equation*}
		the induced map
		\begin{equation*}
			\Map^\psi_{\cat^\otimes}(x,y) \xrightarrow{\sim} \prod_{W \in \mathrm{Orbit}(U')} \Map^{\chi_{[W \subseteq U']} \circ \psi}_{\cat^\otimes}\left(x,y_W\right)
		\end{equation*}
		is an equivalence. 
	\end{enumerate}
	Given a $ \mathcal{T} $-$ \infty $-operad $ (\cat^\otimes,p) $, its \emph{underlying $ \mathcal{T} $-$ \infty$-category} is the fiber product
	\begin{equation*}
	 	\cat := \mathcal{T}^\op \fiberproduct_{\underline{\Fin}_{\mathcal{T},*}} \cat^\otimes .
	\end{equation*} 

	\cite[Definition 2.1.8]{NS22}
	Given a $ \mathcal{T} $-$ \infty $-operad $ (\cat^\otimes,p) $, an edge of $ \cat^\otimes $ is \emph{inert} if it is $ p $-cocartesian over an inert edge of $ \underline{\Fin}_{\mathcal{T},*} $, and it is \emph{active} if it factors as a $ p $-cocartesian edge followed by an edge lying over a fiberwise active edge in $ \underline{\Fin}_{\mathcal{T},*} $. 
\end{defn}

\begin{ex} [Indexing systems]
	\label{ex:indexingsystems} 
	Let us recall that the $ C_p $-$ \E_\infty $-operad is given by $ \mathrm{Com}_{C_p}^\otimes = \underline{\Fin}_{C_p,*} $ the $ \mathcal{O}_{C_p} $-operad corresponding to the maximal indexing system \cite[Example 2.4.7]{NS22}. 
	The minimal indexing system $ \Com_{\mathcal{O}_{C_p}^\simeq}^\otimes $ is a $ C_p $-$ \infty $ operad with underlying category the wide subcategory of $ \underline{\Fin}_{\mathcal{O}_{C_p},*} $ containing those morphisms
	\begin{equation*}
	\begin{tikzcd}[cramped]
		U \ar[d] & Z \ar[l] \ar[d] \ar[r,"m"] & X \ar[d] \\
		V & Y \ar[l] \ar[r,equals] & Y
	\end{tikzcd}	
	\end{equation*}
	where $ m $ is a coproduct of (possibly empty) fold maps. 
	The structure map is the natural inclusion $ {\Com_{\mathcal{O}_{C_p}^\simeq}^\otimes \subseteq \Com_{C_p}^\otimes} $. 
\end{ex}
\begin{defn}
	\cite[Definition 2.2.3]{NS22} 
	Let $ p: \cat^\otimes \to \underline{\Fin}_{\mathcal{T},*} $ be a fibration of $ \mathcal{T} $-$\infty $-operads in which $ p $ is moreover a cocartesian fibration. 
	Then we will call $ \cat^\otimes $ a \emph{$ \mathcal{T} $-symmetric monoidal $ \mathcal{T} $-$\infty $-category}. 
\end{defn}
\begin{recollection}
	\cites[Example 2.4.2]{NS22}[\S9]{norms} \label{recollection:Cpcat_norms}
	The $ C_p $-$ \infty $-category of $ C_p $-spectra is endowed with a $ C_p $-symmetric monoidal structure via the Hill--Hopkins--Ravenel norm functors as follows: Example 2.4.2 \cite{NS22} and \S9 of \cite{norms} define a functor 
	\begin{equation*}
	\begin{split}
		\zeta :=\mathbf{SH}^\otimes \circ \omega_{C_p} \colon & \Span(\Fin_{C_p}) \to \Alg_{\E_1}(\Cat) \\
		& T \mapsto \mathbf{SH}^\otimes \circ \omega_{C_p}(T)		
	\end{split} . 
	\end{equation*} 
	Unravelling definitions, this functor takes
	\begin{equation}\label{eq:Cpcat_classifying_functor}
		\begin{split}
			\zeta  & \colon (C_p \surj C_p/C_p) \mapsto \Spectra^{BC_p} \\
			& \colon (C_p = C_p) \mapsto \Spectra^{BC_p} \\
			&\colon (C_p/C_p = C_p/C_p) \mapsto \Spectra^{C_p} \\
			&\colon \begin{tikzcd}[ampersand replacement=\&]
				 C_p/C_p \ar[d, equals] \& C_p \ar[l] \ar[d] \\
				  C_p/C_p \ar[r, equals] \&  C_p/C_p
			\end{tikzcd} 
			\mapsto 
			\begin{tikzcd}[ampersand replacement=\&]
			 \Spectra^{C_p} \ar[d,"{(-)^e}"] \\ \Spectra^{BC_p}
			 \end{tikzcd} \\
			 & \colon \begin{tikzcd}[ampersand replacement=\&]
				 C_p \ar[d] \ar[r] \& C_p/C_p \ar[d,equals] \\
				  C_p/C_p \ar[r, equals] \&  C_p/C_p
			\end{tikzcd} 
			\mapsto 
			\begin{tikzcd}[ampersand replacement=\&]
			 \Spectra \ar[d,"{N^{C_p}}"] \\ \Spectra^{C_p}
			 \end{tikzcd} \\
			 & \colon \begin{tikzcd}[ampersand replacement=\&]
				 C_p \ar[d] \ar[r,equals] \& C_p \ar[d, equals] \\
				  C_p/C_p  \&  C_p \ar[l]
			\end{tikzcd} 
			\mapsto 
			\begin{tikzcd}[ampersand replacement=\&]
			 \Spectra^{BC_p} \ar[d,"{\id}"] \\ \Spectra^{BC_p}
			 \end{tikzcd}
		\end{split}.
	\end{equation}
	Under Theorem 2.3.9 of \cite{NS22}, this corresponds to a cocartesian fibration $ p: \int \zeta := \left( \underline{\Spectra}^{C_p}\right)^\otimes \to \underline{\Fin}_{C_p,*} $. 
\end{recollection}
In this paper we use the notion of a \emph{$ C_p $-$\E_\infty$-ring} in the sense of Nardin--Shah \cite[Definition 2.2.1]{NS22}. 
\begin{defn}\label{defn:param_alg}
	Let $ \cat^\otimes, \mathcal{D}^\otimes \to \mathcal{O}^\otimes $ be fibrations of $ C_p $-$ \infty $-operads. 
	A $ \mathcal{T} $-functor $ p: \cat^\otimes \to \mathcal{D}^\otimes $ is \emph{a morphism of $ \mathcal{T}$-$ \infty $-operads over $ \mathcal{O} $} if $ p $ takes inert morphisms in $ \cat^\otimes $ to inert morphisms in $ \mathcal{D}^\otimes $.
	Then the category of \emph{$ \cat^\otimes $-algebras in $ \mathcal{D} $}
	\begin{equation*}
		\underline{\Alg}_{\mathcal{O},\mathcal{T}}(\cat,\mathcal{D})
	\end{equation*}
	is the full $ \mathcal{T} $-subcategory of $ \underline{\Fun}_{\mathcal{T}}(\cat,\mathcal{D}) $ on the morphisms of $ \mathcal{T}$-$ \infty $-operads over $ \mathcal{O} $. 
	We write $ {\Alg}_{\mathcal{O},\mathcal{T}}(\cat,\mathcal{D}) $ for the (ordinary) $ \infty $-category of $ \mathcal{T} $-objects in $ \underline{\Alg}_{\mathcal{O},\mathcal{T}}(\cat,\mathcal{D}) $.

	When $ \mathcal{O} $ and/or $ \cat $ are equivalent to $ {\underline{\Fin}_{\mathcal{T},*}} $, we drop them from notation. 

	We write $ \Alg_{\underline{\Fin}_{C_p,*}}\left(\underline{\Fin}_{C_p,*},\left(\underline{\Spectra}^{C_p}\right)^\otimes\right) =: C_p\E_\infty\Alg\left(\Spectra^{C_p}\right) $. 
\end{defn}
\begin{ex}\label{ex:param_calg}
	The category of $ C_p $-$\E_\infty $-rings in $ C_p $-spectra is $ C_p\E_\infty\Alg(\Spectra^{C_p}) $ the space of sections of $ p \colon \left( \underline{\Spectra}^{C_p}\right)^\otimes \to \underline{\Fin}_{C_p,*} $ (Recollection \ref{recollection:Cpcat_norms}) which take inert morphisms to inert morphisms. 

	The inclusion $ \Com_{\mathcal{T}^\simeq}^\otimes \subseteq \Com_{\mathcal{T}}^\otimes $ of Example \ref{ex:indexingsystems} induces a forgetful map
	\begin{equation}\label{eq:forget_paramalg_to_ordinary_calg}
		G \colon \Alg_{\mathcal{T}}\left(\Com_{C_p},\mathcal{D}\right) \to \Alg_{\mathcal{T}}\left(\Com_{\mathcal{O}_{C_p}^\simeq},\mathcal{D}\right) .
	\end{equation}
\end{ex} 
The discussion immediately following \cite[Theorem 4.3.4]{NS22} is summarized by the following result. 
\begin{theorem}\label{thm:operad_forget_is_radjoint}
	Suppose $ p: \cat^\otimes \to \mathcal{O}^\otimes $ is a fibration of $ \mathcal{T} $-$ \infty $-operads, and let $ \mathcal{E}^\otimes \to \mathcal{O}^\otimes $ be a $ \mathcal{T} $-$ \infty $-operad. 
	Then the restriction functor
	\begin{equation*}
		p^* \colon \Alg_{\mathcal{O},\mathcal{T}}(\mathcal{E}) \to \Alg_{\mathcal{O},\mathcal{T}}(\cat,\mathcal{E})
	\end{equation*}
	admits a left adjoint $ p_! $.
\end{theorem}
\begin{defn}\label{defn:param_operadic_LKE}
	Suppose $ p: \cat^\otimes \to \mathcal{O}^\otimes $ is a fibration of $ \mathcal{T} $-$ \infty $-operads, and let $ \mathcal{E}^\otimes \to \mathcal{O}^\otimes $ be a $ \mathcal{T} $-$ \infty $-operad. 
	Let $ A : \cat^\otimes \to \mathcal{E}^\otimes $ be an $ \mathcal{O} $-algebra map. 
	Then the $ \mathcal{O} $-algebra map $ p_!A $ of Theorem \ref{thm:operad_forget_is_radjoint} will be referred to as the \emph{$ \mathcal{T} $-operadic left Kan extension} of $ A $. 
\end{defn}
\begin{rmk}\label{rmk:param_operadic_LKE_specialize} \cite[Remark 4.0.1]{NS22}
	Definition \ref{defn:param_operadic_LKE} specializes to the theory of operadic left Kan extensions of \cite[\S3.1.2]{LurHA} when $ \mathcal{T} = \Delta^0 $. 
\end{rmk}

\section{Normed rings}\label{section:normedrings}
In defining the category of $ C_p $-normed rings, \S\ref{subsection:motivation} guides how we axiomatize the information contained in a $ C_p $-$ \E_\infty $ ring. 
We will see that this information is most naturally captured as the limit of a diagram of $ \infty $-categories (Definition \ref{defn:normedrings}).  
We then exhibit a formula for mapping spaces in normed rings which will be used in the proof of our main theorem (in particular see Proposition \ref{prop:unit_in_normedalg}). 
Finally, we close out this section by showing in Proposition \ref{prop:normedalg_monadic} that the category of normed $ \E_\infty $-rings is monadic over the category of ordinary $ \E_\infty $-algebras. 
\subsection{Preliminaries}
\begin{construction}\label{cons:multiplication_functor}
	Consider the functor
	\begin{equation*}
		\begin{split}
			|-| \colon \mathcal{O}_{C_p} \to \Fin_* \\
			C_p \mapsto \langle p \rangle \\ 
			C_p/C_p \mapsto \langle 1 \rangle
		\end{split}	
	\end{equation*}	
	which takes the underlying set of a set-with-$ C_p $-action. 
	Since $ \Span(\Fin) $ is $ 0 $-semiadditive, the composite $ \mathcal{O}_{C_p} \to \Fin_* \subset \Span(\Fin) $ induces $ \Span\left( \mathcal{O}_{C_p}^\sqcup, fold,all\right) \to \Span(\Fin) $ which restricts to
	\begin{equation*}
	 	m := - \times |-| \colon \Fin_* \times \mathcal{O}_{C_p} \to \Fin_*. 
	\end{equation*}  
	Denote the adjoint of $ m $ by $ m^\dag \colon \Fin_* \to \Fun(\mathcal{O}_{C_p}, \Fin_*) $. 
	Given a symmetric monoidal $ \infty $-category $ q \colon \cat^\otimes \to \Fin_* $, the induced map 
	\begin{equation*}
	 	\Fun_{\Fin_*}\left(\mathcal{O}_{C_p},\cat^\otimes \right)^\otimes := \Fun\left(\mathcal{O}_{C_p},\cat^\otimes \right) \fiberproduct_{\Fun\left(\mathcal{O}_{C_p}, \Fin_* \right),m^\dag} \Fin_* \to \Fin_*
	\end{equation*}  
	is a cocartesian fibration of $ \infty $-operads (cf. \cite[Remark 2.1.3.4]{LurHA}). 
	Since $ \cat^\otimes $ is symmetric monoidal, given any morphism $ h\colon X \to Y $ in $ \mathcal{O}_{C_p} $ and any lift $ \tilde{X} $ of $ |X| $, there is a $ q $-cocartesian morphism $ \tilde{h} $ lifting $ |h| $, so by \cite[Proposition 2.4.4.2]{LurHTT} there is a functor
	\begin{equation*}
		\Fun_{\Fin_*}\left(\mathcal{O}_{C_p},\cat^\otimes \right) \to \Fun\left(\mathcal{O}_{C_p},\cat \right).
	\end{equation*}
	Restriction along $ m $ induces a functor which we also denote by 
	\begin{equation}\label{eq:calg_multiplication_functor}
		m_{(-)} \colon \E_\infty\Alg\left(\cat^\otimes\right) \to \Fun\left(\mathcal{O}_{C_p}, \E_\infty\Alg(\cat^\otimes) \right).
	\end{equation}
	Informally, $ m $ takes an $ \E_\infty $-algebra $ A $ to the $ \mathcal{O}_{C_p} $-diagram $ m_A \colon A^{\otimes p} \to A $. 
\end{construction}
\begin{ntn}
	The prime $ p $ is left implicit in the notation $ m_A $ of Construction \ref{cons:multiplication_functor}, and when $ A $ is understood it may also be dropped from notation. 
\end{ntn}

\begin{rmk}
	The parametrized norm map $ n_A \colon N^{C_p}(A^e) \to A $ is \emph{invariant} with respect to the $ C_p $-action coming from $ A^e $. 
	On the other hand, $ (A^e)^{\otimes p} $ has a $ C_p $-action via cyclic permutations and $ n_A^e $ may also be regarded as a $ C_p $-equivariant map. 
	The reader is warned to remember the distinction between these two $ C_p $-actions; the following observations clarify how these actions interact differently with the structure maps inherent to a $ C_p $-$ \E_\infty $-algebra. 
\end{rmk}
\begin{ntn}\label{ntn:clarifying_actions}
	Let $ A \in \Spectra^{BC_p} $ and let $ \sigma \in C_p $ be a generator. 
	Write $ A^{\otimes^\Delta p} $ for the object in $ \E_\infty\Alg(\Spectra^{BC_p} ) $ with the \emph{diagonal} action, i.e. the composite
	\begin{equation}\label{eq:diagonal_action_tensor}
		\E_\infty\Alg\left(\Spectra \right)^{BC_p} \xrightarrow{R \circ m} \E_\infty\Alg\left(\Spectra \right)^{BC_p \times BC_p} \xrightarrow{\Delta^*} \E_\infty\Alg\left(\Spectra \right)^{BC_p}
	\end{equation}
	where $ m $ is (\ref{eq:calg_multiplication_functor}), $ R $ is restriction along the inclusion $ BC_p \subseteq \mathcal{O}_{C_p} $, and $ \Delta^* $ is restriction along the diagonal $ \Delta: BC_p \to BC_p \times BC_p $. 
	Informally, we regard $ A^{\otimes^\Delta p} $ as being equipped with the $ C_p $-action where $ \sigma $ acts by $ \sigma (a_1 \otimes \cdots \otimes a_p) = \sigma(a_p) \otimes \sigma (a_1) \otimes \cdots \otimes \sigma(a_{p-1}) $. 
	Write $ A^{\otimes^\tau p} $ for the object in $ \E_\infty\Alg(\Spectra^{BC_p} ) $ with the \emph{transposition} action, i.e. the same definition as in (\ref{eq:diagonal_action_tensor}) but with the map $ \{e\} \times \id \colon BC_p \to BC_p \times BC_p $ instead of $ \Delta $. 
	Informally, we regard $ A^{\otimes^\tau p} $ as being equipped with the $ C_p $-action where $ \sigma $ acts by $ \sigma (a_1 \otimes \cdots \otimes a_p) = a_p \otimes a_1 \otimes \cdots \otimes a_{p-1} $. 
\end{ntn}
\begin{obs}\label{obs:clarifying_actions}
	Let $ A \in \Spectra^{BC_p} $. 
\begin{enumerate}[label=(\arabic*)]
	\item The shear endomorphism $ sh := \id_A \otimes \sigma \otimes \cdots \otimes \sigma^{p-1} $ of $ A^{\otimes p} \in \E_\infty \Alg(\Spectra) $ promotes to an equivalence $ A \otimes^\Delta A \to A \otimes^\tau A $ in $ \E_\infty \Alg(\Spectra)^{BC_p} $--in particular it is $ C_p $-equivariant.
	\item Moreover, the Tate diagonal $ A \to (A^{\otimes^{\tau}p} )^{tC_p} $ is equivariant with respect to the given $ C_p $-action on the source and the diagonal $ C_p $-action on the target. 
\end{enumerate}
\end{obs}
\begin{defn}\label{defn:twisted_tate_frob}
	Let $ A \in \E_\infty\Alg\left(\Spectra^{BC_p}\right) $. 
	The \emph{Tate-valued norm} is the $ \E_\infty $-ring map defined by the composite
	\begin{equation*}
		\nu_A: A \xrightarrow{\Delta} \left(A^{\otimes^{\tau} p}\right)^{tC_p} \xrightarrow{sh} \left(A^{\otimes^{\Delta} p}\right)^{tC_p} \xrightarrow{m^{tC_p}} A^{tC_p} 
	\end{equation*}
	where $ \Delta $ is the Tate diagonal of Recollection \ref{rec:tate_diagonal} and $ sh $ is the shear equivalence of Observation \ref{obs:clarifying_actions}. 
	In particular, it is $ C_p $-equivariant with respect to the given action on $ A $, the diagonal $ C_p $-action on $ (A^{\otimes_\sphere^{\tau} p}) $, and the \emph{trivial} action on $ A^{tC_p} $. 
	We regard $ \nu_A $ as a morphism $ A_{hC_p} \to A^{tC_p} $, or equivalently as an object of $ \Fun\left(\mathcal{O}_{C_p}, \E_\infty \Alg(\Spectra)\right) $.
\end{defn}
\begin{rmk}
	Informally, we think of the Tate-valued norm as being $ a \mapsto a \gamma(a) \cdots \gamma^{p-1}(a) $, which is a ring homomorphism modulo transfers. 
	Note that when $ A $ is equipped with the trivial $ C_p $-action, this is simply the ordinary Tate-valued Frobenius (compare Definition IV.1.1 of \cite{nikolaus-scholze}).
\end{rmk}
\subsection{Definition and properties}
We introduce some notation for the indexing category. 
\begin{ntn}\label{ntn:indexing_normedrings}
	Let $ K $ denote the $ \infty $-categorical nerve of the $ 1 $-category
	
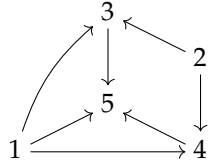
\begin{equation*}
	\begin{tikzcd}[cramped, row sep=tiny]
		& 3 \ar[dd]	& \\
		& & 2 \ar[dd] \ar[lu] \\
		& 5 & \\
		1  \ar[rr] \ar[ru] \ar[ruuu, bend left=15] & & 4 \ar[lu]
	\end{tikzcd}
	\end{equation*}
	in which all triangles and squares commute. 
\end{ntn}
\begin{defn}\label{defn:normedrings}
	Consider the diagram $ \mathcal{N} \colon K \to \Cat_\infty $ where $ K $ is as in Notation \ref{ntn:indexing_normedrings}:  
	\begin{equation}\label{diagram:normedrings}
	\begin{tikzcd}[column sep=small]
		& \Fun\left(\mathcal{O}_{C_p},\E_\infty\Alg\left(\Spectra^{BC_p}\right)\right) \ar[dd,"{\ev_{C_p},\ev_{C_p/C_p}}"]	& \\
		& & \Fun\left(\mathcal{O}_{C_p},\E_\infty\Alg\left(\Spectra^{C_p}\right)\right) \ar[dd,"{\ev_{C_p}, \ev_{C_p/C_p}}"] \ar[lu,"{(-)^e}"'] \\
		&\E_\infty\Alg(\Spectra)^{BC_p \times BC_p} \times \E_\infty\Alg(\Spectra)^{BC_p} & \\
		\E_\infty\Alg\left(\Spectra^{C_p}\right)  \ar[rr,"{N^{C_p}(-^e) \times \id}"'] \ar[ru] \ar[ruuu,"{m \circ (-^e)}", bend left=30] & &\E_\infty\Alg\left(\Spectra^{C_p}\right)^{BC_p} \times \E_\infty\Alg\left(\Spectra^{C_p}\right) \ar[lu,"{(-)^e \times (-)^e}"]
	\end{tikzcd}
	\end{equation}
	where $ m $ is the functor of (\ref{eq:calg_multiplication_functor}). 
	Observe that the right-hand trapezoid of (\ref{diagram:normedrings}) commutes essentially by definition, and the leftmost triangle commutes because $ \left(N^{C_p} A \right)^e \simeq (A^e)^{\otimes p} $. 
	We define the category of \emph{normed} $ C_p $-$ \E_\infty $-rings to be the limit of the diagram 
	\begin{equation}
		N\E_\infty\Alg\left(\Spectra^{C_p}\right) := \lim_K \mathcal{N} .
	\end{equation} 

	There is a canonical forgetful functor $ G' \colon N\E_\infty\Alg\left(\Spectra^{C_p}\right) \to \E_\infty\Alg\left(\Spectra^{C_p}\right)  $ given by the canonical projection to the lower left corner of the diagram (\ref{diagram:normedrings}). 
\end{defn}
\begin{ntn}
	Write $ p_i \colon N\E_\infty\Alg\left(\Spectra^{C_p}\right) \to \mathcal{N}(i) $ for the canonical projection functors. 

	We will often abuse notation and abbreviate an object of $ N\E_\infty\Alg\left(\Spectra^{C_p}\right) $ as a pair $ {(A, n_A: N^{C_2}A \to A)} $ (suppressing the data of the equivalence $ n_A^e \simeq m_{A^e} $). 
\end{ntn}
\begin{rmk}\label{rmk:normedalg_presentable}
	Note that all categories in (\ref{diagram:normedrings}) are presentable and all functors are left adjoints (Lemma \ref{lemma:norm_cocontinuous_alg}), so by \cite[Proposition 5.5.3.13]{LurHTT} we may take the limit in either $ \Pr^L $ or $ \Cat_\infty $. 
\end{rmk}
\begin{prop}\label{prop:normedrings_alt}
	The category $ N \E_\infty\Alg\left( \Spectra^{C_p}\right) $ can be equivalently described as the limit of the diagram
	\begin{equation*}
	\begin{tikzcd}[column sep=small]
		& \Fun\left(\mathcal{O}_{C_p},\E_\infty\Alg(\Spectra)\right) \ar[dd,"{\ev_{C_p},\ev_{*}}"]	& \\
		& & \Fun\left(\mathcal{O}_{C_p},\Ar(\E_\infty\Alg(\Spectra))\right) \ar[dd,"{\ev_{C_p}, \ev_{*}}"] \ar[lu,"{\ev_1}"'] \\
		&\E_\infty\Alg(\Spectra)^{BC_p} \times \E_\infty\Alg(\Spectra) & \\
		\E_\infty\Alg\left(\Spectra^{C_p}\right)  \ar[rr,"{((-)^{\varphi{C_p}}\to (-)^{tC_p} )\circ (N^{C_p} \times \id)}"'] \ar[ru] \ar[ruuu,"{m^{tC_p}:(-^e)^{\otimes p} \to (-^e)}", bend left=30] & &\Ar\left(\E_\infty\Alg\left(\Spectra^{BC_p}\right)\right) \times \Ar\left(\E_\infty\Alg\left(\Spectra\right)\right)  \ar[lu,"{\ev_1}"']
	\end{tikzcd}
	\end{equation*}
\end{prop}
\begin{proof}
	Follows from Corollary \ref{cor:eqvtalg_recollement}. 
\end{proof}

\begin{rmk} \label{rmk:CpEinftyringdata}
	Recall the description of a limit of $ \infty $-categories given by Corollary 3.3.3.2 of \cite{LurHTT}. 
	Combining this with the description of mapping spaces in $ \Alg_{\E_\infty}\left(\Spectra^{C_p}\right) $ which is a consequence of Corollary \ref{cor:eqvtalg_recollement}, we may equivalently characterize a normed $\E_\infty $-ring as the data of a $ \E_\infty $-algebra $ A $ in $ \Spectra^{C_p} $ plus the data of a factorization $ n_A^{\varphi C_p} $ in $ \E_\infty\Alg(\Spectra) $ and a 2-cell making the diagram
	\begin{equation}\label{eq:tatefrobenius_lift}
	\begin{tikzcd}[column sep=huge]
		A \ar[d,"\Delta"'] \ar[r, dotted,"{n_A}"]  & A^{\varphi C_2} \ar[d,"\alpha"] \\
		\left(A^{ \otimes^\tau p}\right)^{tC_p} \ar[r,"{m^{tC_p} \circ (sh)}"] & A^{tC_p} 
	\end{tikzcd}
	\end{equation}
	commute such that, considered as a morphism $ g : \Delta \to \alpha $, $ g $ is equivariant with respect to the given $ C_p $-action on the source and the trivial $ C_p $-action on the target. 
	Note that the composite of the left arrow followed by the lower arrow in (\ref{eq:tatefrobenius_lift}) is the Tate-valued norm of Definition \ref{defn:twisted_tate_frob}. 

	When $ C_p $ acts trivially on $ A^e $, it suffices to produce the 2-cell (\ref{eq:tatefrobenius_lift}). 
	More formally, given a choice of multiplication map $ m: A^{\otimes^\Delta p} \to A $, by Corollary \ref{cor:eqvtalg_recollement} we have an equivalence of fibers
	\begin{equation*}
	\begin{split}
		\fib_{\{m\}}\left(\hom_{\E_\infty\Alg(\Spectra^{C_p})^{BC_p}}\left(N^{C_p}(A^e), A\right) \xrightarrow{(-)^e} \hom_{\E_\infty\Alg\Spectra^{BC_p \times BC_p}}\left((A^e)^{\otimes^{\Delta} p}, A^e\right)\right) \\
		\simeq \fib_{\{m^{tC_p}\}}	\left\{\hom_{\E_\infty\Alg^{\Delta^1 \times BC_p}}\left( 
			\begin{tikzcd}[ampersand replacement=\&, column sep=tiny]
				A^e \ar[d,"\Delta"] \ar[d] \&  A^{\varphi C_p} \ar[d,"\alpha"] \\
				\left(A^{\otimes_\sphere^{\tau} p}\right)^{tC_p}, \& A^{tC_p} 
			\end{tikzcd}
		\right)	\xrightarrow{\ev_1} \hom_{\E_\infty\Alg^{BC_p}}\left(\left(A^{\otimes_\sphere^\tau p}\right)^{tC_p}, A^{tC_p} \right) \right\}
	\end{split}.
	\end{equation*}
\end{rmk}
\begin{obs}
	[Morphism spaces in normed algebras]\label{obs:morphisms_normedrings}
	Let $ s, t \colon K \to \int \mathcal{N} $ be objects in the limit $ N\E_\infty\Alg(\Spectra^{C_p}) $, which we identify as spaces of coCartesian sections of $ \int \mathcal{N} \to K $ where $ \mathcal{N} \colon K \to \Cat_\infty $ is the diagram defining (\ref{diagram:normedrings}). 
	Now by definition of a limit of $ \infty $-categories, we may write the space of morphisms from $ s $ to $ t $ in $ N\E_\infty $ as $ \lim_{k}\hom_{F(k)}(s(k), t(k)) $. 
 
	Unravelling definitions, given a pair $ (A, n_A \colon N^{C_p}A \to A), (B, n_B \colon N^{C_p} B \to B ) $ in the limit of (\ref{diagram:normedrings}), the morphism space $ \Hom_{N\E_\infty}((A, n_A \colon N^{C_p}A \to A),(B, n_B \colon N^{C_p} B \to B ))  $ is computed as the limit of the diagram 
	\begin{equation}\label{eq:morphisms_normedrings_diagram}
	\begin{tikzcd}[column sep=small]
		& \hom\left( \morphism{A^e}{}{(A \otimes A)^{tC_p}}, \morphism{B^e}{}{(B \otimes B)^{tC_p}} \right) \ar[dd] & \\
		& & \mathclap{{\hom\left(\morphism{A}{\Delta}{\left(A^{\otimes p}\right)^{tC_p}}, \morphism{B}{\Delta}{\left(B^{\otimes p}\right)^{tC_p}} \right)}  \fiberproduct_{\hom\left(\morphism{A}{\Delta}{\left(A^{\otimes p}\right)^{tC_p}}, \morphism{B^{\varphi C_p}}{\Delta}{B^{tC_p}} \right)} {\hom\left(\morphism{A^{\varphi C_p}}{\Delta}{A^{tC_p}}, \morphism{B^{\varphi C_p}}{\Delta}{B^{tC_p}} \right)}} \ar[dd] \ar[lu, bend right=30] \\ 
		& \mathclap{\hom\left( (A \otimes A)^{tC_p}, (B \otimes B)^{tC_p} \right) \times \hom\left( A^{tC_p}, B^{tC_p} \right)} & \\
		\mathclap{\hom_{\E_\infty\Alg\left(\Spectra^{C_p}\right)}(A,B)} {} \ar[rr,shorten <=11ex] \ar[ru,shorten >=3ex] \ar[ruuu,"{m:(-^e)^{\otimes p} \to (-^e)}"',near end, bend left=60] & & {\hom\left(\morphism{A}{\Delta}{\left(A^{\otimes p}\right)^{tC_p}}, \morphism{B}{\Delta}{\left(B^{\otimes p}\right)^{tC_p}} \right)} \times {\hom\left(\morphism{A^{\varphi C_p}}{\Delta}{A^{tC_p}}, \morphism{B^{\varphi C_p}}{\Delta}{B^{tC_p}} \right)}	\ar[lu,shorten >=5ex]	
	\end{tikzcd}.
	\end{equation}
\end{obs}

\begin{prop}\label{prop:normedalg_monadic}
	The forgetful functor $ G' \colon N\E_\infty\Alg\left(\Spectra^{C_p}\right) \to \E_\infty\Alg\left(\Spectra^{C_p}\right) $ of Definition \ref{defn:normedrings} is monadic. 
\end{prop}
\begin{proof}
	The functor $ G' $ is conservative by inspection. 

	Recall our notation $ p_i \colon N\E_\infty\Alg(\Spectra^{C_p}) \to \mathcal{N}(i) $ for the canonical projection functors.  
	Now suppose given a simplicial object $ A \colon \Delta^\op \to N\E_\infty\Alg\left( \Spectra^{C_p} \right) $ which is $ G' $-split. 
	Then in particular $ p_0\circ A $ is a colimit diagram of $ \E_\infty $-algebras in $ \Spectra^{C_p} $. 
	Since the norm preserves all colimits of algebras by Lemma \ref{lemma:norm_cocontinuous_alg}, $ p_4 \circ A \simeq (N^{C_p} \times \id) \circ p_0 \circ A $ is a colimit diagram. 
	By \cite[Corollary 5.1.2.3(2)]{LurHA} applied to $ S = \mathcal{O}_{C_p} $, $ p_2 \circ A $ is a colimit diagram. 
	Now by Remark \ref{rmk:normedalg_presentable} and Proposition 5.1.2.2(2) of \emph{loc. cit.} applied to $ S = K $, $ A $ is a colimit diagram in $ N\E_\infty\Alg(\Spectra^{C_p}) $, and said colimit is preserved by $ G' $.
	Thus $ G'$ is monadic by the Barr--Beck--Lurie theorem \cite[Theorem 4.7.3.5]{LurHA}.
\end{proof}

\section{Comparing \texorpdfstring{$ C_p $}{Cp}-\texorpdfstring{$ \E_\infty $}{E∞} and normed rings}
In \S\ref{subsection:comparison_functor} we write down a functor from $ C_p $-$ \E_\infty $-algebras to normed $ C_p $-algebras. 
Our proof strategy will be to show that the comparison functor of Corollary \ref{cor:operadic_diagrammatic_comparison} exhibits both $ C_p $-$ \E_\infty $-algebras and normed $ \E_\infty $-algebras as categories of algebras over the \emph{same} monad on $ \E_\infty\Alg\left(\Spectra^{C_p}\right) $, then appeal to a variant of the Barr--Beck--Lurie theorem. 
In \S\ref{subsection:param_monoidal_envelope} we exhibit a formula for the free $ C_p $-$ \E_\infty $-algebra on an $ \E_\infty $-algebra, then we show in \S\ref{subsection:mainthm_proof} that it induces an equivalence. 

\subsection{A comparison functor}\label{subsection:comparison_functor}
Since a normed $ \E_\infty $-ring is a priori less data than a $ C_p $-$ \E_\infty $-algebra, it is most natural to define a `forgetful' functor from the latter to the former. 
In order to write down the functor, we need to unpack the definition of a $ C_2 $-$ \E_\infty $-algebra.  
\begin{ntn}
	\label{ntn:paramfinC2set_breakdown}
	Observe that $ \Span\left(\Fin_{C_p}^{/V}\right) $ is $ 0 $-semiadditive \cite[Lemma C.3]{norms} and define $ \underline{\Span}(\Fin_{C_p,*}) $ to be the colimit of the functor $ V \mapsto \Span\left(\Fin_{C_p}^{/V}\right) $ \cite[Corollary 3.3.4.3]{LurHTT}. 
	Since $ 0 $-semiadditive $ \infty $-categories are closed under all colimits \cite[Corollary 5.4]{MR4133704}, $ \underline{\Span}(\Fin_{C_p}) $ is $ 0 $-semiadditive. 
	Moreover, notice that there is an inclusion $ \underline{\Fin}_{C_p,*} \subset \underline{\Span}(\Fin_{C_p}) $. 

	Let $ \delta : J \to \underline{\Fin}_{C_p,*} $ be a diagram. 
	Under the equivalence of \cites[Theorem 4.1]{MR4133704}[Lemma C.4]{norms} the diagram $ \delta $ classifies a functor $ \Span(J^\sqcup, \mathrm{fold}, \mathrm{all}) \to \underline{\Span}(\Fin_{C_p}) $ which evidently restricts to
	\begin{equation}\label{eq:restriction_finGset}
	 	\iota_J \colon \Fin_* \times J \to \underline{\Fin}_{C_p,*}. 
	\end{equation}  
	When $ J = \Delta^0 $ and $ \delta $ is the inclusion of a single object $ T \in \underline{\Fin}_{C_p,*} $, we write $ \iota_T $. 

	Consider the diagrams $ \alpha_2, \alpha_3 \colon \mathcal{O}_{C_p} \to \underline{\Fin}_{C_p,*} $
	\begin{equation*}
		\begin{tikzcd}[cramped]
			C_p \ar[r,equals] \ar[d] & C_p \ar[d] \ar[r] & C_p/C_p \ar[d,equals] \\
			C_p/C_p \ar[r,equals] & C_p/C_p \ar[r,equals] & C_p/C_p
		\end{tikzcd}	
		\qquad \qquad
		\begin{tikzcd}[cramped]
			{C_p^{\sqcup p}} \ar[r,equals] \ar[d,"\nabla"] & {C_p^{\sqcup p}} \ar[d,"\nabla"] \ar[r,"{\nabla}"] & C_p \ar[d,equals] \\
			C_p \ar[r,equals] & C_p \ar[r,equals] & C_p
		\end{tikzcd}
	\end{equation*}	
	resp., where $ C_p $ acts on $ C_p^{ \sqcup p}  $ by permuting the terms of the disjoint union. 
	The preceding discussion shows that there are functors
	\begin{equation}\label{eq:restriction_finGset_mors}
	 	\iota_{\alpha_i} := - \times \alpha_i \colon \Fin_* \times \mathcal{O}_{C_p} \to \underline{\Fin}_{C_p,*}. 
	\end{equation}  
	By a similar discussion to that of Construction \ref{cons:multiplication_functor}, the $ \iota_{(-)} $ induce functors 
	\begin{equation}\label{eq:restrictionparamalg_is_alg}
		\iota_J \colon C_p\E_\infty\Alg\left(\Spectra^{C_p} \right) \to \Fun \left(J, \E_\infty \Alg\left(\Spectra^{C_p} \right) \right) .
	\end{equation}
\end{ntn}

\begin{construction}\label{cons:restrictions}
	Recall that the category of $ C_p $-$ \E_\infty $-algebras $ C_p\E_\infty\Alg $ is given by sections of the fibration $ \left(\underline{\Spectra}^{C_p}\right)^\otimes $ (Definition \ref{defn:param_alg}). 
	By (\ref{eq:restrictionparamalg_is_alg}), restricting to certain subcategories of $ \underline{\Fin}_{\C_p,*} $ gives functors: 
	\begin{enumerate}[label=(\alph*)]
		\item $ \gamma_1 \colon C_p\E_\infty\Alg\left(\Spectra^{C_p}\right) \to \E_\infty\Alg\left(\Spectra^{C_p}\right) $ given by restriction along $ \iota_T $ of (\ref{eq:restriction_finGset}) for $ {T = [C_p/C_p = C_p/C_p]} $. 
		\item $ \gamma_4 \colon C_p\E_\infty\Alg\left(\Spectra^{C_p}\right) \to \E_\infty\Alg\left(\Spectra^{C_p}\right) \times \E_\infty\Alg\left(\Spectra^{C_p}\right) $ given by restriction along $ \iota_T \times \iota_S $ of (\ref{eq:restriction_finGset}) for $ T = [C_p \surj C_p/C_p] $ and $ {S = [C_p/C_p = C_p/C_p]} $. 
		\item $ \gamma_5 \colon C_p\E_\infty\Alg\left(\Spectra^{C_p}\right) \to \E_\infty\Alg(\Spectra)^{BC_p \times BC_p} \times \E_\infty\Alg(\Spectra)^{BC_p} $ given by restriction along $ {\iota_T \times \iota_S} $ of (\ref{eq:restriction_finGset}) for $ {T = \left[C_p^{\sqcup p} \xrightarrow{\nabla} C_p\right]} $ and $ {S = [C_p = C_p]} $, resp. 
		Note that in the former case, $ C_p $ acts by permuting the factors of $  C_p^{\sqcup p} \simeq C_p \times C_p $ cyclically.
		\item $ \gamma_2 \colon C_p \E_\infty\Alg \to \Fun\left(\mathcal{O}_{C_p},\E_\infty\Alg(\Spectra^{C_p}) \right) $ given by restriction along $ \iota_{\alpha_2} $ of (\ref{eq:restriction_finGset_mors}). 
		\item $ \gamma_3 \colon C_p \E_\infty\Alg \to \Fun\left(\mathcal{O}_{C_p},\E_\infty\Alg(\Spectra^{BC_p}) \right) $ given by restriction along $ \iota_{\alpha_3} $ of (\ref{eq:restriction_finGset_mors}).
	\end{enumerate}
\end{construction}
\begin{prop}\label{prop:restrictions_compatibilities}
	The functors of Construction \ref{cons:restrictions} are related in the following way:
	\begin{enumerate}[label=(\alph*)]
		\item There is an equivalence $ m \circ \gamma_1^e \simeq \gamma_3 $.
		\item There is an equivalence $ \ev_{C_p} \times \ev_{C_p/C_p} \circ \gamma_3 \simeq \gamma_5 $. 
		\item There is an equivalence $ (((-)^e)^{\otimes 2} \times (-)^e )\circ \gamma_1 \simeq \gamma_5 $. 
		\item There is an equivalence
		$ \left(\ev_{C_p} \times \ev_{C_p/C_p}\right) \circ \gamma_2 \simeq \gamma_4 $ of functors $ C_p\E_\infty\Alg\left(\Spectra^{C_p} \right) \to \E_\infty\Alg\left(\Spectra^{C_p} \right)^{\times 2} $. 
		\item There is an equivalence
		$ \left(N^{C_p} \times \id\right) \circ \gamma_1 \simeq \gamma_4 $ of functors $ C_p\E_\infty\Alg\left(\Spectra^{C_p} \right) \to \E_\infty\Alg\left(\Spectra^{C_p} \right)^{\times 2} $. 
		\item There is a commutative diagram
		\begin{equation*}
			\begin{tikzcd}[column sep=tiny, row sep=small]
				C_p\E_\infty \Alg\left(\Spectra^{C_p}\right) \ar[rd,"{\gamma_3}"'] \ar[rr,"{\gamma_2}"] & & \Fun\left(\mathcal{O}_{C_p}, \E_\infty \Alg\left(\Spectra^{C_p}\right)\right) \ar[ld,"{(-)^e}"] \\
				& \Fun\left(\mathcal{O}_{C_p}, \E_\infty \Alg(\Spectra^{BC_p})\right) &
			\end{tikzcd}
		\end{equation*}
	\end{enumerate}
\end{prop}
\begin{cor}\label{cor:operadic_diagrammatic_comparison}
	There is a canonical functor 
	\begin{equation*}
	 	\gamma \colon C_p\E_\infty\Alg\left(\Spectra^{C_p}\right) \to N\E_\infty\Alg\left(\Spectra^{C_p}\right).
	\end{equation*} 
\end{cor}
\begin{proof}
	The functors of Construction \ref{cons:restrictions} may be regarded as 
	\begin{equation*}
	 	\gamma_i : C_p\E_\infty\Alg\left(\Spectra^{C_p}\right) \to \mathcal{N}(i)
	\end{equation*} 
	where $ \mathcal{N} : K \to \Cat_\infty $ is as in Definition \ref{defn:normedrings} and Notation \ref{ntn:indexing_normedrings}. 
	Proposition \ref{prop:restrictions_compatibilities} shows that the functors $ \gamma_i $ commute with the structure maps in the diagram $ \mathcal{N} $. 
	By definition of a homotopy limit, the $ \gamma_i $ assemble to the desired functor $ \gamma $. 
\end{proof}
\begin{proof} [Proof of Proposition \ref{prop:restrictions_compatibilities}]
\begin{enumerate}[label=(\alph*)]
	\item Consider the diagram $ T:= \mathcal{O}_{C_p} \times \Delta^1 \to \underline{\Fin}_{C_p,*} $ 
	\begin{equation}\label{diagram:underlying_of_fold}
		\begin{tikzcd}[column sep=small, row sep=small,cramped]
			{C_p/C_p^{\sqcup p}} && {C_p/C_p} \\
			& {C_p^{\sqcup p}} && {C_p} \\
			{C_p/C_p} && {C_p/C_p} \\
			& {C_p} && {C_p}
			\arrow[from=1-1, to=1-3,"\nabla"]
			\arrow[from=1-1, to=3-1]
			\arrow[from=3-1, to=3-3,equals]
			\arrow[from=1-3, to=3-3,equals]
			\arrow[from=2-2, to=1-1]
			\arrow[from=2-2, to=4-2,crossing over]
			\arrow[from=4-2, to=3-1]
			\arrow["\lrcorner"{anchor=center, pos=0.125, rotate=-90}, draw=none, from=2-2, to=3-1]
			\arrow[from=2-4, to=1-3]
			\arrow[from=2-4, to=4-4,equals]
			\arrow[from=4-2, to=4-4,equals]
			\arrow[from=4-4, to=3-3]
			\arrow["\lrcorner"{anchor=center, pos=0.125, rotate=-90}, draw=none, from=2-4, to=3-3]
			\arrow[from=2-2, to=2-4,crossing over,"\nabla",near start]
		\end{tikzcd}.
	\end{equation}
	Note that $ (m\circ \gamma_1)^e \simeq m \circ (\gamma_1^e) $. 
	Now notice that $ m\circ \gamma_1 $ is given by restriction along $ \iota_{\nabla_{C_p/C_p}} $ (i.e. the back face), while restriction along the front face implements $ \gamma_3 $. 
	We may regard $ \iota_T $ (Notation \ref{ntn:paramfinC2set_breakdown}) as a natural transformation $ \beta: (m \circ \gamma_1)^e \implies \gamma_3 $ by (\ref{eq:Cpcat_classifying_functor}). 
	Since the morphisms from the back face to the front face of (\ref{diagram:underlying_of_fold}) are inert, $ \beta $ is a natural equivalence.   

	\item This is evident. 

	\item Follows from (a) and (b). 

	\item This is evident from the definitions of $ \alpha_2 $ and $ \gamma_4 $. 

	\item Consider the morphism $ w \colon \Delta^1 \to \underline{\Fin}_{C_p,*} $
	\begin{equation*}
		\begin{tikzcd}[ampersand replacement=\&,cramped]
			 C_p/C_p \ar[d, equals] \& C_p \ar[l] \ar[d] \ar[r,equals] \& C_p \ar[d] \\
			  C_p/C_p \ar[r, equals] \&  C_p/C_p \ar[r, equals] \& C_p/C_p
		\end{tikzcd} .
	\end{equation*} 
	Notice that $ w $ is inert and recall that morphisms of operads take inert morphisms to inert morphisms. 
	Because a morphism in $ \left(\underline{\Spectra}^{C_p}\right)^\otimes $ factors canonically as a $ p $-cocartesian morphism and a fiberwise morphism, by definition of $ \zeta $ (\ref{eq:Cpcat_classifying_functor}) we see that restriction along $ \iota_w $ gives an equivalence $ N^{C_p}(\gamma_1^e) \simeq \pi_1 \gamma_4 $. 

	\item Now consider the diagram $ T:= \mathcal{O}_{C_p} \times \Delta^1 \to \underline{\Fin}_{C_p,*} $ 
	\begin{equation}\label{diagram:index_underlying_of_norm}
		\begin{tikzcd}[column sep=small, row sep=small]
			{C_p} && {C_p/C_p} \\
			& {C_p \times C_p} && {C_p} \\
			{C_p/C_p} && {C_p/C_p} \\
			& {C_p} && {C_p}
			\arrow[from=1-1, to=1-3]
			\arrow[from=1-1, to=3-1]
			\arrow[from=3-1, to=3-3,equals]
			\arrow[from=1-3, to=3-3,equals]
			\arrow[from=2-2, to=1-1,"{\pi_1}",near start]
			\arrow[from=2-2, to=4-2,crossing over]
			\arrow[from=4-2, to=3-1]
			\arrow["\lrcorner"{anchor=center, pos=0.125, rotate=-90}, draw=none, from=2-2, to=3-1]
			\arrow[from=2-4, to=1-3]
			\arrow[from=2-4, to=4-4,equals]
			\arrow[from=4-2, to=4-4,equals]
			\arrow[from=4-4, to=3-3]
			\arrow["\lrcorner"{anchor=center, pos=0.125, rotate=-90}, draw=none, from=2-4, to=3-3]
			\arrow[from=2-2, to=2-4,crossing over,"{g=\pi_2}",near start]
		\end{tikzcd}
	\end{equation}
	considered as an inert morphism (in fact, $ \ev_1 $-cocartesian) from the back face $ \alpha_2 $ (Notation \ref{ntn:paramfinC2set_breakdown}) to the front face. 

	Identifying the underlying set of $ C_p $ with $ \{0, 1, \ldots, p-1\} $, notice that the shear equivalence 
	\begin{equation*}
	\begin{split}
	 	sh \colon \{0,1,\ldots, p-1\} \times C_p \to C_p \times C_p \\
		(a, b) \mapsto (a+b, b)
	\end{split}
	\end{equation*} 
	which is equivariant with respect to the diagonal $ C_p $-action on the target and the action by $ C_p $ on the second factor on the source. 
	The shear map identifies $ g $ with the fold map $ \nabla : C_p^{\sqcup p} \to C_p $, i.e. there is a commutative diagram
	\begin{equation*}
	\begin{tikzcd}[row sep=tiny, cramped]
		C_p \times C_p \ar[rd,"{\pi_2 = g}"] & \\
									& C_p  \\
		\{0,1,\ldots,p-1\} \times C_p \ar[ru,"{\pi_2 = \nabla}"'] \ar[uu,"sh","{\rotatebox{90}{$\sim$}}"']
	\end{tikzcd}. 
	\end{equation*}
	Thus we see that the shear map induces an equivalence $ \iota_g \simeq \iota_{\alpha_3} \simeq \gamma_3  $.  

	Now restriction along $ \iota_T $ (Notation \ref{ntn:paramfinC2set_breakdown}) gives a natural transformation $ \beta $
	\begin{equation*}
	\begin{tikzcd}[row sep=tiny, column sep=large]
												&	\Fun\left( \mathcal{O}_{C_p}, \E_\infty \Alg\left(\Spectra^{C_p}\right)\right) \ar[dd,"{(-)^e}"]  \\
		C_p\E_\infty\Alg\left(\Spectra^{C_p}\right) \ar[ru,"{\iota_{\alpha_2}\simeq \gamma_2}"] \arrow[rd,"{\iota_{g}}"',""{name=T}]{} & \\
												&	\Fun\left( \mathcal{O}_{C_p},\E_\infty \Alg\left(\Spectra^{BC_p}\right) \right) \arrow[Rightarrow,from=1-2,to=T,"\beta"]
	\end{tikzcd}.
	\end{equation*}
	Since the back-to-front arrows in (\ref{diagram:index_underlying_of_norm}) are inert, $ \beta $ is a natural equivalence. 
	\qedhere
\end{enumerate}
\end{proof}

\subsection{A parametrized monoidal envelope}\label{subsection:param_monoidal_envelope}
To apply the Barr--Beck--Lurie theorem \cite[Proposition 4.7.3.22]{LurHA}, we will need to show that $ \gamma $ of the free $ C_p $-$ \E_\infty $-algebra on an $ \E_\infty $-algebra $ A $ computes the free normed algebra on $ A $.
A general strategy for understanding free $ C_p $-$ \E_\infty $ algebras is outlined in Remark 4.3.6 of \cite{NS22}; we introduce the ingredients first, then outline the strategy in Recollection \ref{recollection:freealg_colimit}. 
Then we apply the aforementioned general strategy to exhibit a formula for the free $ C_p $-$ \E_\infty $-algebra on an $ \E_\infty $-algebra $ A $ in Theorem \ref{theorem:freenormedalg_formula}. 

\begin{defn}\label{defn:operadic_envelope}
	\cite[Definition 2.8.4 \& Notation 2.8.3]{NS22}
	Let $ \mathcal{O}^\otimes $ be a $ \mathcal{T} $-operad. 
	let
	\begin{equation*}
		\Ar^{act}_{\mathcal{T}}(\mathcal{O}^\otimes) := \mathcal{T}^\op \times_{\Ar(\mathcal{T}^\op)} \Ar^{act}(\mathcal{O}^\otimes)
	\end{equation*}
	where $  \Ar^{act}(\mathcal{O}^\otimes) $ is the $ \mathcal{T} $-full subcategory on the active morphisms. 

	Suppose given a fibration of $ \mathcal{T}$-operads $ p \colon \mathcal{C}^\otimes \to \mathcal{O}^\otimes $. 
	The \emph{$ \mathcal{O} $-monoidal envelope of $ \cat^\otimes $} is
	\begin{equation*}
		\pi : \mathrm{Env}_{\mathcal{O},\mathcal{T}}(\mathcal{C})^\otimes := \mathcal{C}^\otimes \times_{\mathcal{O}^\otimes} \Ar^{act}_{\mathcal{T}}(\mathcal{O}^\otimes) \to \mathcal{O}^\otimes .
	\end{equation*}
	When $ \mathcal{O}  = \underline{\Fin}_{\mathcal{T},*} $ we drop it from notation. 
\end{defn}
\begin{rmk}\label{rmk:underlying_of_envelope}
	More informally, an object of $ \mathrm{Env}_{\mathcal{O},\mathcal{T}}(\mathcal{C})^\otimes $ is a pair $ (c, g \colon p(c) \to o) $ where $ c \in \mathcal{C}^\otimes $, $ o \in \mathcal{O}^\otimes $, $ g $ is a fiberwise active arrow in $ \mathcal{O}^\otimes $. 
	The forgetful map $ \pi $ takes this tuple to $ o $. 
	By \cite[Remark 2.8.5]{NS22}, the underlying $ \mathcal{T} $-$ \infty $-category of $ \mathrm{Env}_{\mathcal{O},\mathcal{T}}(\mathcal{C})^\otimes $ is $ \mathrm{Env}_{\mathcal{O},\mathcal{T}}(\mathcal{C}) \simeq \Com_{\mathcal{T}^\simeq, act}^\otimes $. 
\end{rmk}
\begin{prop}\label{prop:alg_extension_to_envelope}
	\cite[Proposition 2.8.7]{NS22} 
	Let $ p:  \cat^\otimes \to \mathcal{O}^\otimes $ be a fibration of $ \mathcal{T} $-$ \infty $-operads, and let $ \mathcal{D}^\otimes \to \mathcal{O}^\otimes $ be a cocartesian fibration of $ \mathcal{T} $-$ \infty $-operads. 
	Let $ i: \cat^\otimes \subseteq \mathrm{Env}_{\Com_{\mathcal{T}}}\left(\cat\right)^\otimes $ denote the inclusion of $ \cat^\otimes $ into its monoidal envelope. 
	Then there is an adjunction
	\begin{equation*}
	\begin{tikzcd}
		i_! : \Alg_{\mathcal{O}, \mathcal{T}}(\cat, \mathcal{D}) \ar[r,shift left=.4ex] & \Alg_{\mathcal{O}, \mathcal{T}}(\mathrm{Env}_{\mathcal{O}, \mathcal{T}}\left(\cat\right)^\otimes, \mathcal{D}) : i^*  \ar[l,shift left=.4ex]
	\end{tikzcd}
	\end{equation*}
	and $ i_! $ has essential image the full subcategory of the right-hand side given by $ \Fun^\otimes_{\mathcal{O}, \mathcal{T}}(\mathrm{Env}_{\mathcal{O}, \mathcal{T}}\left(\cat\right)^\otimes, \mathcal{D}) $. 
\end{prop}

\begin{recollection}\label{recollection:freealg_colimit}
	Consider $ \mathcal{O} = \underline{\Fin}_{\mathcal{T},*} $, $ \mathcal{E} = \left(\underline{\Spectra}^{C_2}\right)^\otimes $ (Notation \ref{ntn:indexing_normedrings}), and $ \cat = \Com_{\mathcal{T}^\simeq} $ (Example \ref{ex:indexingsystems}) in Theorem \ref{thm:operad_forget_is_radjoint}. 
	Then there is an adjunction
	\begin{equation*}
		F: \E_\infty\Alg(\Spectra^{C_2} ) \leftrightarrows C_2\E_\infty\Alg(\Spectra^{C_2} ) : G.
	\end{equation*}
	where $ G $ is from (\ref{eq:forget_paramalg_to_ordinary_calg}). 
	By Remarks 2.8.5 \& 4.3.6 of \cite{NS22}, the free $ C_2 $-$ \E_\infty $-algebra $ F(A) $ on an $ \E_\infty $-algebra in $ \Spectra^{C_2} $ is computed by the $ C_2 $-left Kan extension of $ i_! A^\otimes \colon \mathrm{Env}_{\mathcal{T}}\left(\Com_{\mathcal{T}^\simeq}\right)^\otimes \to \underline{\Spectra}^{C_2} $ along the structure map $ \pi : \mathrm{Env}_{\mathcal{T}}\left(\Com_{\mathcal{T}^\simeq}\right)^\otimes \to \Com_{\mathcal{T}}^\otimes $, where $ i_! $ is from Proposition \ref{prop:alg_extension_to_envelope}. 
\end{recollection}
\begin{theorem}\label{theorem:freenormedalg_formula}
	Let $ A \in \E_\infty\Alg\left(\Spectra^{C_p}\right)  $ (also see Lemma \ref{lemma:Einftyalg_2ways}) and consider the adjunction $ F \dashv G $ of Recollection \ref{recollection:freealg_colimit}. 
	\begin{enumerate}[label=(\arabic*)]
	\item The underlying $ C_p $-spectrum of the free $ C_p $-$ \E_\infty $ algebra $ F(A) $ on $ A $ is given by (via the recollement of Proposition \ref{prop:eqvtspectrarecollement})
	\begin{equation}\label{eq:freenormedalg_formula}
	F(A) \simeq
	\begin{tikzcd}
		&  A^{\varphi C_p} \otimes A^e_{hC_p} \ar[d,"{s_A \otimes \nu_A}"] \\
		A^e \ar[r,mapsto] & A^{tC_p}
	\end{tikzcd}
	\end{equation}
	where $ u $ is the unit, $ s_A : A^{\varphi C_p} \to A^{tC_p} $ is the structure map, and $ \nu_A $ is the twisted Tate-valued Frobenius (Definition \ref{defn:twisted_tate_frob}). 

	\item There is a canonical $ \E_\infty $ ring map $ \eta_A \colon A \to GF(A) $ given by $ \id_{A^{\varphi C_p}} \otimes (\eta_{A^e_{hC_p}}:\sphere^0 \to A^e_{hC_p}) $ on geometric fixed points and the identity on underlying.
	\end{enumerate}
\end{theorem}
\begin{proof}
	By Recollection \ref{recollection:freealg_colimit}, the $ C_p $-$ \E_\infty $-algebra $ F(A) $ may be computed as the $ C_p $-left Kan extension of $ i_! A^\otimes \colon \mathrm{Env}_{\mathcal{T}}(\Com_{\mathcal{T}^\simeq})^\otimes \to \left(\underline{\Spectra}^{C_2}\right)^{\otimes} $ along the structure map $ \pi : \mathrm{Env}_{\mathcal{T}}(\Com_{\mathcal{T}^\simeq})^\otimes \to \Com_{\mathcal{T}}^\otimes $. 
	
	Denote $ x = [C_p/C_p = C_p/C_p] \in \Com_{\mathcal{T}}^\otimes $. 
	In particular, the $ C_p $-spectrum underlying $ F(A) $ may be computed as the $ C_p $-left Kan extension of $ {i_! A^\otimes_{\underline{x}} \colon \left(\Com_{\mathcal{T}^\simeq, act}^\otimes\right)_{\underline{x}} \to \underline{\Spectra}^{C_p}} $ along the structure map $ \pi_{\underline{x}} : \left(\Com_{\mathcal{T}^\simeq, act}^\otimes\right)_{\underline{x}} \simeq \mathrm{Env}_{\mathcal{T}}(\Com_{\mathcal{T}^\simeq})^{\otimes}_{\underline{x}} \to \left(\Com_{\mathcal{T}}^\otimes\right)_{\underline{x}} \simeq \mathcal{T}^\op $. 

	Let $ I = \left(\Com_{\mathcal{T}^\simeq, act}^\otimes\right)_{\underline{x}} \xrightarrow{q} \mathcal{T}^\op $ be shorthand for our indexing diagram and write $ I_{C_p} $ and $ I_{C_p/C_p} $ for the respective fibers (not parametrized fibers). 
	We will write $ F|_{C_p} $ for the restriction of a diagram $ F $ defined on $ I $ to $ I_{C_p} $.

	By definition of a $ C_2 $-left Kan extension and Definition 5.2 of \cite{Shah18}, we seek a $ \mathcal{T}^\op $-initial lift 
	\begin{equation}\label{eq:param_undercategory}
	\begin{tikzcd}
		& \underline{\Fun}_{\mathcal{T}^\op}\left( I \star_{\mathcal{T}^\op} \mathcal{T}^\op, \underline{\Spectra}^{C_p} \right) \ar[d] \\
		\mathcal{T}^\op \ar[r,"{i_! A^\otimes_{\underline{x}}}"'] \ar[ru,dotted,"\widetilde{F(A)}"] & \underline{\Fun}_{\mathcal{T}^\op}\left( I , \underline{\Spectra}^{C_p} \right)
	\end{tikzcd}.
	\end{equation} 
	Informally, a lift $ \widetilde{F(A)} $ of (\ref{eq:param_undercategory}) is the data of 
	\begin{itemize}
		\item a cocartesian section $ F(A) \colon \mathcal{T}^\op \to \underline{\Spectra}^{C_p} $ 
		\item a morphism $ \beta $ from $ \left. i_! A^\otimes_{\underline{x}}\right|_{C_p} $ to the constant $ I_{C_p} $-indexed diagram at $ F(A)(C_p) $ in $ \Spectra^{BC_p} $
		\item a morphism $ \alpha $ from $ \left. i_! A^\otimes_{\underline{x}} \right|_{C_p/C_p} $ to the constant $ I_{C_p/C_p} $-indexed diagram at $ F(A)(C_p/C_p) $ in $ \Spectra^{C_p} $
		\item Choose a functor $ R \colon I_{C_p/C_p} \to I_{C_p} $ classified by the map $ C_p \surj C_p/C_p $. 
		Then we require the data of an equivalence $ (\alpha)^e \simeq \beta \circ R $ of natural transformations. 
	\end{itemize}
	Now notice that the diagram $ i_! A^\otimes_{\underline{x}}|_{C_p} $ is defined on $ \left(\Com_{\mathcal{T}^\simeq,act}^\otimes\right)_{\underline{x},C_p} $, which has a final object $ {[C_p = C_p]} $. 
	Thus for $ \widetilde{F(A)}(C_p) $ to be an initial object of the $ C_p $-fiber of (\ref{eq:param_undercategory}), we must have $ {F(A)^e \simeq A^e} $. 
	An initial object of the $ C_p/C_p $-fiber of (\ref{eq:param_undercategory}) is equivalently an object $ F(A) : \mathcal{T}^\op \to \underline{\Spectra}^{C_p} $ representing the functor
	\begin{equation*}
	\begin{split}
		\underline{\Spectra}^{C_p} &\to \Spc \\
		B &\mapsto \hom_{\underline{\Fun}_{\mathcal{T}^\op}(I,\underline{\Spectra}^{C_p})}\left(i_! A^\otimes_{\underline{x}}, q^* B\right) .
	\end{split}
	\end{equation*}
	By a similar argument to our earlier discussion of morphisms in categories of cocartesian sections (Observation \ref{obs:morphisms_normedrings}) and Proposition \ref{prop:eqvtspectrarecollement}, the space of morphisms from the diagram $ i_! A^\otimes_{\underline{x}} $ to $ q^*B $ sits in a fiber sequence 
	\begin{equation*}
	\begin{tikzcd}
		\fib\left(\hom_{\Fun\left(I|_{C_p/C_p},Ar(\Spectra)\right)}\left(\morphism{(i_! A^\otimes_{\underline{x}}|_{C_p/C_p})^{\varphi C_p}}{}{(i_! A^\otimes_{\underline{x}}|_{C_p/C_p})^{tC_p}},\morphism{(q^*B|_{C_p/C_p})^{\varphi C_p}}{}{q^*B|_{C_p/C_p}^{tC_p}}\right)\to \hom_{\Fun\left(I|_{C_p},\Spectra\right)}\left((i_! A^\otimes_{\underline{x}}|_{C_p})^{tC_p},q^*B|_{C_p}^{tC_p}\right)\right) \ar[d] \\
		\hom_{\Fun_{\mathcal{T}^\op}\left(I,\underline{\Spectra}^{C_p}\right)}\left(i_! A^\otimes_{\underline{x}}, q^*B\right) \simeq 
		\hom_{\Fun\left(I|_{C_p/C_p},\underline{\Spectra}^{C_p}\right)}\left(i_! A^\otimes_{\underline{x}}|_{C_p/C_p}, q^*B|_{C_p/C_p}\right) \ar[d,"{(-)^e}"] \\
		{\hom_{\Fun_{BC_p}\left(I|_{C_p},\underline{\Spectra}^{BC_p}\right)}(i_! A^\otimes_{\underline{x}}|_{C_p}, q^*B|_{C_p})} 
	\end{tikzcd}.
	\end{equation*}
	By the previous discussion, we have $ {\hom_{\Fun_{BC_p}\left(I|_{C_p},\underline{\Spectra}^{BC_p}\right)}(i_! A^\otimes_{\underline{x}}|_{C_p}, q^*B^e)} \simeq \hom_{\Spectra^{BC_p}}(A^e, B^e) $. 
	Thus we see that for $ \widetilde{F(A)}(C_p/C_p) $ to be an initial object of the $ C_p/C_p $-fiber of (\ref{eq:param_undercategory}), it suffices to take $ F(A)^{\varphi C_p} $ to be the colimit of the diagram
	\begin{equation*}
		 \left(i_! A^\otimes_{\underline{x}}\right)^{\varphi C_p} \colon \left(\Com_{\mathcal{T}^\simeq,act}^\otimes\right)_{\underline{x}, C_p/C_p} \to \Spectra . 
	\end{equation*}
	By Lemma \ref{lemma:freeC2alg_colimdiagram}, the $ C_p $-left Kan extension of $ i_! A^\otimes_{\underline{x}} $ along $ \pi_{\underline{x}} : \Env( \Com_{\mathcal{T}^\simeq})_{\underline{x}} \to \left(\Com_{\mathcal{T}}\right)_{\underline{x}} $ is computed on $ C_p $ geometric fixed points by (\ref{eq:freenormedalg_formula}). 
	
	The existence of the unit $ \eta_A $ follows from monadicity (Proposition \ref{prop:forgetC2ringmap_monadic}), and its exact form follows from tracing through the definition of $ C_p $-left Kan extension. 
\end{proof}
\begin{warning}
	The $ G $-$ \infty $-category of $ G $-spectra $ \underline{\Spectra}^G $ is \emph{not} an example of the $ G $-category of objects of Construction \ref{cons:Tcat_of_objects}. 
	Thus many of the techniques to compute $ G $-left Kan extensions of \cite{Shah18} do not apply to our proof of Theorem \ref{theorem:freenormedalg_formula}. 
\end{warning}
\begin{lemma}\label{lemma:freeC2alg_colimdiagram}
	Consider the fiber $ \left(\Com_{\mathcal{T}^\simeq}^\otimes\right)_{act, C_p/C_p} $ (not parametrized fiber) over $ C_p/C_p $ of $ \left(\Com_{\mathcal{T}^\simeq}^\otimes\right)_{act} $. 
	The inclusion 
	\begin{equation*}
		\iota \colon BC_p \sqcup * \inj \left(\Com_{\mathcal{T}^\simeq}^\otimes\right)_{act, C_p/C_p} 
	\end{equation*} 
	of the full subcategory spanned by the object 
	\begin{equation*}
	\begin{tikzcd}[cramped]
		C_p \sqcup C_p/C_p \ar[r] \ar[d]  & C_p/C_p \ar[d,equals] \\
		C_p/C_p \ar[r,equals] & C_p/C_p
	\end{tikzcd}
	\end{equation*} is cofinal.
\end{lemma}
\begin{proof}
	Since we consider the fiber over $ C_p/C_p $, the target is always $ C_p/C_p $ and we omit it throughout the following proof. 
	Observe that there is a pullback diagram of simplicial sets
	\begin{equation*}
	\begin{tikzcd}[cramped]
		BC_p \sqcup * \ar[r,"\iota"] \ar[d] & \left(\Com_{\mathcal{T}^\simeq}^\otimes\right)_{act, C_p/C_p} \ar[d,"Q"] \\
		\{\langle 2 \rangle, \langle 1 \rangle, \{*, 2\} \} \ar[r] & \mathrm{Sub}
	\end{tikzcd}.
	\end{equation*}
	By Lemma \ref{lemma:projection_orbits_cocartesian} and \cite[Proposition 4.1.2.15]{LurHTT}, $ Q $ is smooth. 
	By Remark 4.1.2.10 of \emph{loc. cit.} and Lemma \ref{lemma:cofinality}, $ \iota $ is cofinal. 
\end{proof}

\begin{recollection}\label{rec:finset_coproduct_operad}
	Recall the category $ \mathrm{Sub} $ of \cite[Definition 2.2.3.2]{LurHA} which was defined to have
	\begin{enumerate}[label=(\alph*)]
		\item objects of $ \mathrm{Sub} $ are triples $ (\langle n\rangle,S,T) $ where $ S, T \subseteq \langle n \rangle $ such that $ S \cup T = \langle n \rangle $ and $ S \cap T = \langle n \rangle $.
		\item a morphism of $ \mathrm{Sub} $ from $ (\langle n\rangle,S,T) $ to $ (\langle n'\rangle,S',T') $ is a pointed map $ f : \langle n\rangle \to \langle n' \rangle $ such that $ f(S) \subseteq S' $ and $ f(T) \subseteq T' $.
	\end{enumerate}
	The $ \infty $-category $ \mathrm{Sub} $ is an $ \infty $-operad by Proposition 2.2.3.5 of \emph{loc. cit.} applied to $ \cat^\otimes = \mathcal{D}^\otimes = \Fin_* $. 
	Write $ \pi : \mathrm{Sub} \to \Fin_* $ for the structure map. 
\end{recollection}
\begin{lemma}\label{lemma:cofinality}
	Let $ \mathcal{A} \subseteq \mathrm{Env}_{\Fin_*}(\mathrm{Sub}) $ denote the full subcategory on the object $ (a,\pi(a) \to \langle 1 \rangle) $ for $ a = (\langle 2 \rangle, \langle 1 \rangle, \{2,*\} ) $.
 	Then the inclusion
	\begin{equation*}
		\iota \colon \mathcal{A} \longrightarrow \Env_{\Fin_*} (\mathrm{Sub})
	\end{equation*}
	is cofinal.
\end{lemma}
\begin{proof}
	We verify criterion (2) of \cite[Theorem 4.1.3.1]{LurHTT}. 
	Observe that for every object $ (\langle n\rangle,S,T) $ of $ \mathrm{Sub} $, there is a unique morphism $ (\langle n\rangle,S,T) \to (\langle 2 \rangle, \langle 1 \rangle, \{2,*\} ) $ which sends all $ s \in S \setminus\{*\} $ to $ 1 $ and all $ t \in T \setminus \{*\} $ to $ 2 $.  
\end{proof}
\begin{lemma}\label{lemma:projection_orbits_cocartesian}
	There is a functor $ Q \colon \Com_{\mathcal{O}_{C_p}^{\simeq}} \to \mathrm{Sub} \simeq \Fin_* \boxplus \Fin_* $ which takes a $ C_p $-set $ T $ to its set of $ C_p $-orbits grouped by orbit type, i.e. $ Q : C_p \mapsto (\langle 1 \rangle, \langle 1 \rangle, \{*\}) $ and $ Q : C_p/C_p \mapsto (\langle 1 \rangle, \{*\}, \langle 1 \rangle) $. 
	The functor $ Q $ is a coCartesian fibration classified by the functor
	\begin{equation*}
		\begin{split}
			\mathrm{Sub} &\to \Cat \\
			(\langle n\rangle,S,T) &\mapsto \bigsqcup_{S} BC_p \sqcup \bigsqcup_T * 
		\end{split}. 
	\end{equation*}
\end{lemma}

\begin{construction}\label{cons:calg_map_out_freeparam_calg}
	Let $ A \in \E_\infty\Alg(\Spectra^{C_p}) $. 
	Given any $ (B, n_B \colon N^{C_p} B \to B ) \in N\E_\infty\Alg\left(\Spectra^{C_p}\right) $, there is a canonical map
	\begin{equation*}
		f: \hom_{\E_\infty\Alg(\Spectra^{C_p})}(A, B) \to \hom_{N\E_\infty\Alg(\Spectra^{C_p}) }(\gamma F(A), (B, n_B \colon N^{C_p} B \to B )). 
	\end{equation*}
	By Observation \ref{obs:morphisms_normedrings} and Corollary \ref{cor:eqvtalg_recollement}, we may define $ f $ `componentwise.'
	Denote $ \hom_{\mathcal{N}(i)}(p_iF(A), p_i(B)) $ by $ M_i $ (Definition \ref{defn:normedrings}). 
	We have 
	\begin{equation*}
	\begin{split}
		 M_0 = {\hom_{\E_\infty\Alg(\Spectra^{C_p})}\left(F(A), B \right)} &\simeq \hom_{\E_\infty\Alg} \left(
		A,B \right) \fiberproduct_{\hom\left(A^{tC_p}, B^{tC_p}\right)} 
		\hom \left(
		\begin{tikzcd}[cramped,column sep=small,ampersand replacement=\&]
			 A^{e}_{hC_p} \ar[d] \& B^{\varphi C_p} \ar[d] \\
			 A^{tC_p}, \& B^{tC_p}
		\end{tikzcd} \right) \\
		&=: M_0' \times_{M^t} M_0''		
	\end{split}.
	\end{equation*}
	Take the identity on $ M_0' $ and define $ \hom_{\E_\infty\Alg(\Spectra^{C_p})}(A, B) \to M_0'' $ to be the composite
	\begin{equation*}
	 	f_0''  \colon \hom_{\E_\infty\Alg(\Spectra^{C_p})}(A, B) \xrightarrow{N^{C_p}} \hom_{\E_\infty\Alg(\Spectra^{C_p})}\left(N^{C_p}A, N^{C_p}B\right) \xrightarrow{g} M' 
	\end{equation*} 
	where $ g $ takes $ h: N^{C_p}A \to N^{C_p}B $ to the outermost trapezoid in the commutative diagram
	\begin{equation*}
	\begin{tikzcd}
		A^e \ar[d,"\Delta"] \ar[rr,"h"] && B^{e} \ar[d,"\Delta"] \ar[rd,"{n_B^{\varphi C_p}}"] & \\
		((A^e)^{\otimes p})^{tC_p} \ar[rr,"{(h^{\otimes p})^{tC_p}}"] \ar[d,"{m^{tC_p}}"'] && ((B^e)^{\otimes p})^{tC_p} \ar[rd, "{n_B^{tC_p}}"] & B^{\varphi C_p} \ar[d] \\
		A^{tC_p} \ar[rrr,"{h^{tC_p}}"] && & B^{tC_p}
	\end{tikzcd}
	\end{equation*}
	where by definition $ n_B^{tC_p} \simeq m_{B^e}^{tC_p} $ and the lower trapezoid is the Tate construction $ (-)^{tC_p} $ on $ m $ of (\ref{eq:calg_multiplication_functor}) applied to $ h $. 
	Clearly $ f_0' $ and $ f_0'' $ lift canonically to a functor $ f_0 : \hom_{\E_\infty\Alg(\Spectra^{C_p})}(A, B) \to M_0 $. 

	Take $ f_4: \hom_{\E_\infty\Alg(\Spectra^{C_p})}(A, B) \to M_4 $ to be the product $ N^{C_p}(-^e) \times \id $. 
	The map $ f_4 $ clearly lifts to $ f_2: \hom_{\E_\infty\Alg(\Spectra^{C_p})}(A, B) \to M_2 $ and is identified canonically with $ (N^{C_p}(-^e) \times \id) \circ f_0 $.  
	Since $ F(A)^e \simeq A^e $, we may define $ f_3 $ and $ f_5 $ as $ m_{(-)^e} $ and $ (-^e)^{\otimes 2} \times (-)^e $, respectively. 

	The $ f_i $ assemble to give the desired map. 
\end{construction}

\subsection{Proof of main theorem}\label{subsection:mainthm_proof}
Equipped with an explicit description of the free $ C_p $-$ \E_\infty $-algebra on an $ \E_\infty $-algebra $ A $ in $ C_p $-spectra from the previous section, here we show that $ \gamma $ of the free $ C_p $-$ \E_\infty $-algebra on an $ \E_\infty $-algebra $ A $ computes the free normed algebra on $ A $ using our description of mapping spaces in the category of normed $ \E_\infty $-algebras. 
The main result then follows from an application of the Barr--Beck--Lurie theorem. 
\begin{theorem} \label{thm:mainthm} 
	The functor $ \gamma \colon C_p\E_\infty\Alg\left(\Spectra^{C_p}\right) \to N\E_\infty\Alg\left(\Spectra^{C_p}\right)$ of Corollary \ref{cor:operadic_diagrammatic_comparison} is an equivalence. 
\end{theorem}
\begin{proof}
	Consider the diagram of forgetful functors 
	\begin{equation}\label{diagram:forgetforget_is_forget}
	\begin{tikzcd}[column sep=tiny]
		C_p\E_\infty\Alg\left(\Spectra^{C_p}\right) \ar[rr,"\gamma"] \ar[rd,"G"'] & & N\E_\infty\Alg\left(\Spectra^{C_p}\right) \ar[ld,"G'"] \\
		& \E_\infty\Alg\left(\Spectra^{C_p}\right) & 
	\end{tikzcd}
	\end{equation} 
	where $ G' $ is from Definition \ref{defn:normedrings} and $ G $ is (\ref{eq:forget_paramalg_to_ordinary_calg}). 
	The diagram (\ref{diagram:forgetforget_is_forget}) evidently commutes. 

	The functor $ G $ is monadic by Proposition \ref{prop:forgetC2ringmap_monadic}. 
	The functor $ G' $ is monadic by Proposition \ref{prop:normedalg_monadic}.
	Now for any $ A \in \E_\infty\Alg\left(\Spectra^{C_2}\right) $, the unit $ A \to \gamma F(A) $ of Theorem \ref{theorem:freenormedalg_formula} induces an equivalence $ F'(A) \simeq \gamma F(A) $ by Corollary \ref{cor:free_algebras_agree}.  
	The result follows from \cite[Proposition 4.7.3.16]{LurHA}. 
\end{proof}

\begin{prop}\label{prop:unit_in_normedalg}
	Let $ A $ be an $ \E_\infty $-ring in $ \Spectra^{C_p} $ and let $ (B, n_B \colon N^{C_p}B \to B) $ be a normed $ \E_\infty $-ring in $ \Spectra^{C_p} $. 
	Then precomposition with the $ \E_\infty $-map $ \eta_A : A \to GF(A) $ of Theorem \ref{theorem:freenormedalg_formula} induces an equivalence of morphism spaces
	\begin{equation}\label{eq:unit_in_normedalg}
	\begin{tikzcd}
		{\Hom_{N\E_\infty}\left(\left(\gamma F(A), n_{F(A)} \colon N^{C_p}\gamma F(A) \to \gamma F(A)\right),\left(B, n_B \colon N^{C_p} B \to B \right)\right)} \ar[d,"{G'}"] \ar[dd, bend left=75,looseness=2,"{\rotatebox{90}{$\sim$}}"] \\
		\Hom_{\E_\infty}(G'\gamma F(A), G'(B)) \ar[d,"{\eta^*}"] \\
		 \Hom_{\E_\infty}(A, G'(B)).
	\end{tikzcd}
	\end{equation}
	where $ G' $ is the forgetful functor of Definition \ref{defn:normedrings} and $ \gamma $ is the functor of Corollary \ref{cor:operadic_diagrammatic_comparison}. 
	That is, $ \eta_{(-)} $ is a unit for the functors $ (\gamma \circ F, G') $ in the sense of \cite[Definition 5.2.2.7]{LurHTT}. 
\end{prop}
\begin{cor}\label{cor:free_algebras_agree}
	The natural transformation $ \eta_{(-)} $ exhibits $ \gamma \circ F $ as a left adjoint to $ G' $. 
\end{cor}
\begin{proof}
	Follows from \cite[Proposition 5.2.2.8]{LurHTT} and Proposition \ref{prop:unit_in_normedalg}. 
\end{proof}
\begin{proof} [Proof of Proposition \ref{prop:unit_in_normedalg}] 
	By Observation \ref{obs:morphisms_normedrings}, the space of morphisms\linebreak $ {\Hom_{N\E_\infty}((F(A), n_{F(A)} \colon N^{C_p}F(A) \to F(A)),(B, n_B \colon N^{C_p} B \to B ))} $ is computed by the limit of the diagram 
	\begin{equation}\label{diagram:freeparam_hom_normedalg}
	\begin{tikzcd}[column sep=small]
		& \hom\left( \morphism{A^e}{}{(A^{\otimes p})^{tC_p}}, \morphism{B^e}{}{(B^{\otimes p})^{tC_p}} \right) \ar[dd] & \\
		& & \mathclap{{\hom\left(\morphism{A}{\Delta}{\left(A^{\otimes p}\right)^{tC_p}}, \morphism{B}{\Delta}{\left(B {\otimes p}\right)^{tC_p}} \right)}  \fiberproduct_{\hom\left(\morphism{A}{\Delta}{\left(A^{\otimes p}\right)^{tC_p}}, \morphism{B^{\varphi C_p}}{\Delta}{B^{tC_p}} \right)} {\hom\left(\morphism{A^{\varphi C_p} \otimes A_{hC_p} }{s_A \otimes \nu_A}{A^{tC_p}}, \morphism{B^{\varphi C_p}}{\Delta}{B^{tC_p}} \right)}} \ar[dd] \ar[lu, bend right=30] \\ 
		& \mathclap{\hom\left( (A^{\otimes p})^{tC_p}, (B^{\otimes p})^{tC_p} \right) \times \hom\left( A^{tC_p}, B^{tC_p} \right)} & \\
		\mathclap{\hom_{\E_\infty\Alg\left(\Spectra^{C_p}\right)}(F(A),B)} {} \ar[rr,shorten <=13ex] \ar[ru,shorten=1.5ex] \ar[ruuu,"{m:(-^e)^{\otimes p} \to (-^e)}"',near end, bend left=50] & & {\hom\left(\morphism{A}{\Delta}{\left(A^{\otimes p}\right)^{tC_p}}, \morphism{B}{\Delta}{\left(B^{\otimes p}\right)^{tC_p}} \right)} \times {\hom\left(\morphism{A^{\varphi C_p} \otimes A_{hC_p}}{\Delta}{A^{tC_p}}, \morphism{B^{\varphi C_p}}{\Delta}{B^{tC_p}} \right)}	\ar[lu,shorten >=5ex]	
	\end{tikzcd}.
	\end{equation}
	where all morphisms are computed in either $ \E_\infty\Alg\left(\Spectra^{BC_p}\right) $ or $ \Ar \E_\infty\Alg\left(\Spectra^{BC_p}\right) $. 
	Notice that 
	\begin{equation*}
		\hom_{\E_\infty\Alg\left(\Spectra^{C_p}\right)}(F(A),B) \simeq \hom_{\E_\infty\Alg\left(\Spectra^{C_p}\right)}(A,B) \fiberproduct_{\hom\left(A^{tC_p},B^{tC_p}\right) } \hom \left( \morphism{A_{hC_p} }{\nu_A}{A^{tC_p}}, \morphism{B^{\varphi C_p}}{\Delta}{B^{tC_p}} \right)
	\end{equation*}
	and moreover the composite
	\begin{equation*}
		\hom_{N\E_\infty\Alg}(F(A), B) \xrightarrow{G'} \hom_{\E_\infty\Alg\left(\Spectra^{C_p}\right)}(F(A),B) \xrightarrow{\pi_1} \hom_{\E_\infty\Alg\left(\Spectra^{C_p}\right)}(A,B)
	\end{equation*}
	 is equivalent to $ \eta^* \circ G' $. 
	Unravelling definitions, we see that given a point $ f \in \hom_{\E_\infty\Alg\left(\Spectra^{C_p}\right)}(A,B) $, the fiber of $ \eta^* \circ G' $ over $ f $ is given by the space of fillings of the below diagram to a commutative diagram $ \mathcal{O}_{C_p} \times (\Delta^1)^{\times 2} \to \E_\infty\Alg(\Spectra) $: 	
	\begin{equation*}
	\begin{tikzcd}[column sep=small, row sep=small]
		A^e \ar[rr,"{f^e}"] \ar[dd] \ar[rd] & & B^e \ar[dd] \ar[rd,"{n_B}"] & \\
			& A^e_{hC_p}  \ar[rr, dotted,crossing over] & & B^{\varphi C_p} \ar[dd] \\
		{(A^{\otimes p})^{tC_p}} \ar[rr,"{(f^{\otimes p})^{tC_p}}",near end] \ar[rd] & & {( B^{\otimes p})^{tC_p}} \ar[rd] & \\
		& A^{tC_p} \ar[rr,"{f^{tC_p}}"] & & B^{tC_p} \ar[from=2-2,to=4-2,crossing over,"{\nu_{A^e}}", near end]
	\end{tikzcd}
	\end{equation*} 
	wherein all but the top and front face are given. 
	This space is contractible by the adjunction
	\begin{equation*}
		(-)^{\mathrm{triv}} \colon \Spectra \leftrightarrows \Spectra^{C_p} \colon (-)_{hC_p} . 	
	\end{equation*} 
	Now by Construction \ref{cons:calg_map_out_freeparam_calg}, $ \eta^* \circ G' $ admits a right inverse $ f $, hence it is surjective on connected components. 
	Thus the result follows. 
\end{proof}

\begin{prop}\label{prop:forgetC2ringmap_monadic}
	The forgetful functor $ G: C_p\E_\infty\Alg\left(\Spectra^{C_p}\right) \to \E_\infty\Alg\left(\Spectra^{C_p}\right) $ of (\ref{eq:forget_paramalg_to_ordinary_calg}) is monadic. 
\end{prop}
\begin{proof}
	We consider the commuting triangle of forgetful functors
	\begin{equation*}
	\begin{tikzcd}[column sep=tiny]
		C_p\E_\infty\Alg\left(\Spectra^{C_p}\right) \ar[rr,"G"] \ar[rd] & & \E_\infty\Alg\left(\Spectra^{C_p}\right) \ar[ld] \\
		& \Spectra^{C_p} & 
	\end{tikzcd}. 
	\end{equation*} 
	The upper horizontal arrow is given by restricting along the $ C_p $ operadic inclusion $ \Com_{\mathcal{T}^\simeq}^\otimes \inj \Com_{\mathcal{T}}^\otimes $ and applying the equivalence of Lemma \ref{lemma:Einftyalg_2ways}. 
	By \cite[Corollary 5.1.5]{NS22}, the diagonal arrows are monadic; in particular by the Barr--Beck--Lurie theorem the right diagonal arrow is conservative. 
	By \cite[Theorem 4.3.4]{NS22} applied to the morphism $ \Com_{\mathcal{T}^\simeq}^\otimes \inj \Com_{\mathcal{T}}^\otimes $, the upper horizontal arrow admits a left adjoint. 
	Thus the result follows from \cite[Proposition 4.7.3.22.]{LurHA}. 
\end{proof}
\begin{lemma}
	\label{lemma:Einftyalg_2ways}
	For the minimal indexing system $ \Com_{\mathcal{T}^\simeq} $ (Example \ref{ex:indexingsystems}), we have a canonical identification of $ \Com_{\mathcal{T}^\simeq} $-algebras in $ \left(\underline{\Spectra}^{C_p}\right)^\otimes $ with $ \mathcal{O}_{C_p}^\op $-families of $ \E_\infty $-algebras in spectra, or equivalently $ \E_\infty $-algebras in $ \Spectra^{C_p} $. 
\end{lemma}
\begin{proof}
	Compare \cite[Corollary 2.4.15]{NS22} and \cite[Example 2.1.3.5]{LurHA}. 

	Let $ A :  \Com_{\mathcal{T}^\simeq}^\otimes \to \cat^\otimes $ be a section of $ p: \cat^\otimes \to \underline{\Fin}_{\mathcal{T},*} $. 
	Notice that $ A $ is a $ \Com_{\mathcal{T}^\simeq} $-algebra if and only if $ A $ is $ \mathcal{T} $-right Kan extended from the full subcategory of $\Com_{\mathcal{T}^\simeq}^\otimes $ spanned by coproducts of $ [C_p/C_p = C_p/C_p] $. 
	The result follows from \cite[Proposition 4.3.2.15]{LurHA}. 
\end{proof}

\section{Applications \& examples}\label{section:app_ex}
\subsection{Examples}
The example which will be used in the author's upcoming work on real ($ C_2 $-equivariant) trace theories is 
\begin{theorem}\label{thm:constMackey_is_normed}
	Let $ k $ be a discrete commutative ring. 
	The constant $ C_p $-Mackey functor $ \underline{k} $ on $ k $ uniquely acquires the structure of a $ C_p $-$ \E_\infty $-ring. 
\end{theorem}
\begin{proof}
In view of Theorem \ref{thm:mainthm}, it suffices to show that $ \underline{k} $ can be lifted to an object of Definition \ref{defn:normedrings}. 
Note that the isotropy separation sequence for $ \underline{k} $ is
\begin{equation*}
\begin{tikzcd}
	k^{C_p} \ar[d,"\tau_{\geq 0}"'] \ar[r] & \underline{k}^{\varphi C_p} \ar[d] \\
	k^{hC_p} \ar[r] & k^{tC_p} 
\end{tikzcd}.
\end{equation*}
The left vertical arrow is a connective cover; hence so is the right vertical arrow and $ \underline{k}^{\varphi C_p} = \tau_{\geq 0} k^{tC_p} $. 
Note that $ \tau_{\geq 0} k^{tC_p} $ is an $ \E_\infty $-ring in spectra because the Tate construction and connective cover are lax symmetric monoidal functors. 
By Theorem \ref{thm:mainthm} and Remark \ref{rmk:CpEinftyringdata}, it suffices to exhibit a commutative diagram
\begin{equation*}
\begin{tikzcd}
	k \ar[d,"\text{Tate diagonal}"'] \ar[r, dotted,"{n_{\underline{k}}}"] & \underline{k}^{\varphi C_p} \ar[d,"\alpha"] \\
	\left(k^{\otimes p}\right)^{tC_p} \ar[r,"{m^{tC_p}}"] & k^{tC_p} 
\end{tikzcd}.
\end{equation*} 
The dotted arrow and 2-cell making the diagram commute exist up to contractible choice because the inclusion of connective $ \E_\infty $-algebra spectra into all $ \E_\infty $-algebra spectra admits a right adjoint \cite[Proposition 7.1.3.13]{LurHA}, and our assumption that $ k $ is connective. 
\end{proof}

There is a natural class of equivariant $ C_p $-spectra for which the data of (\ref{eq:tatefrobenius_lift}) is no extra data at all. 
\begin{recollection}
	The $ \infty $-category $ \Spectra^{C_p, \Borel} $ of \emph{Borel} $ C_p $-spectra is the image of $ \Spectra^{BC_p} $ under the fully faithful right adjoint to the `underlying' spectrum functor of Proposition \ref{prop:eqvtspectrarecollement}. 
	Write $ C_p\E_\infty\Alg^{\Borel}\left(\Spectra^{C_p} \right) $ for the pullback
	\begin{equation*}
		C_p\E_\infty\Alg\left(\Spectra^{C_p} \right) \fiberproduct_{\Spectra^{C_p}} \Spectra^{C_p, \Borel} .
	\end{equation*}
	In words, this is the category of $ C_p $-$ \E_\infty $-algebras whose underlying $ C_p $-spectrum is Borel. 
	A $ C_p $-spectrum is Borel if and only if the structure map $ A^{\varphi C_p} \to A^{tC_p} $ is an equivalence. 
\end{recollection}
In view of the expected correspondence between the theories of $ N_\infty $-algebras of Blumberg--Hill and the $ C_p $-$ \E_\infty $-algebras of Nardin--Shah, we have the following analogue of \cite[Theorem 6.26]{MR3406512}. 
\begin{prop}\label{prop:Borel_are_normed} 
	Every Borel $ \E_\infty$-algebra in $ C_p $-spectra admits an essentially unique structure of a $ C_p\E_\infty $-algebra. 
	More precisely, there is an equivalence of categories
	\begin{equation*}
		\E_\infty\Alg^{\Borel}\left(\Spectra^{C_p} \right) \xrightarrow{\exists} C_p\E_\infty\Alg^{\Borel}\left(\Spectra^{C_p} \right)
	\end{equation*}
	with inverse the forgetful functor $ G $ (\ref{eq:forget_paramalg_to_ordinary_calg}). 
\end{prop}
This result may also be regarded as a special case of \cite[Proposition 3.3.6]{HilmanThesis}. 
\begin{proof}
	Let $ A \in \E_\infty\Alg^{\Borel}\left(\Spectra^{C_p} \right) $. 
	Then by Theorem \ref{thm:mainthm}, it suffices to produce a lift
	\begin{equation*}
	\begin{tikzcd}
		A^e \ar[d,"\Delta"] \ar[r,dotted,"{\exists ?}"] & A^{\varphi C_p} \ar[d,"{s_A}"] \\
		\left( A^{\otimes p}\right)^{tC_p} \ar[r,"{m^{tC_p}}"] & A^{tC_p}
	\end{tikzcd}
	\end{equation*}
	which is functorial in $ A $. 
	By definition of Borel spectra, $ s_A $ is an equivalence, so the space of choices of $ \dasharrow $ and a 2-cell making the diagram commute is contractible. 
\end{proof}
\begin{cor}\label{cor:real_bordism_normed}
	The real bordism spectrum $ MU_{\R} $ admits a unique refinement to $ C_2 $-$ \E_\infty $-algebra.
\end{cor}
\begin{proof}
	Follows from \cite[Theorem 4.1(1)]{MR1808224} and Proposition \ref{prop:Borel_are_normed}.
\end{proof}
\begin{prop}\label{prop:norms_are_normed}
	Let $ B \in \E_\infty\Alg \left(\Spectra\right) $ be an $ \E_\infty $-algebra. 
	Then $ N^{C_p}B $ admits a canonical structure of a $ C_p $-$ \E_\infty $-algebra. 
	That is, there is a factorization
	\begin{equation*}
	\begin{tikzcd}[cramped]
		&  N\E_\infty\Alg\left(\Spectra^{C_p}\right) \ar[d,"{G'}"] \\
		\E_\infty\Alg \left(\Spectra\right) \ar[ru,dotted,"\exists",bend left=30] \ar[r,"{N^{C_p}}"] & \E_\infty \Alg\left(\Spectra^{C_p}\right)
	\end{tikzcd}.
	\end{equation*}
\end{prop}
\begin{proof}
	Then by Theorem \ref{thm:mainthm}, it suffices to produce a dotted arrow and a commutative diagram $ \mathcal{O}_{C_p} \times \Delta^1 \to \E_\infty\Alg\left(\Spectra\right) $
	\begin{equation*}
	\begin{tikzcd}
		B^{\otimes p} \ar[d,"\Delta"'] \ar[r,dotted,"{\exists ?}"] & B \ar[d,"{\Delta}"] \\
		\left( B^{\otimes p^2} \right)^{tC_p} \ar[r,"{m^{tC_2}}"] & (B^{\otimes p})^{tC_p}
	\end{tikzcd}
	\end{equation*}
	which is functorial in $ B $. 
	We can choose the dotted arrow to be $ m_B $ and a commutative diagram to be functorial in $ B $ because the Tate diagonal is functorial. 
\end{proof}

\subsection{Real motivic invariants}
Here, we briefly recall that algebras with $ C_2 $-actions naturally give rise to motivic invariants valued in genuine $ C_2 $-spectra. 
These real motivic invariants and their associated real trace theories provided the impetus for this work. 

We only provide brief sketches of the required constructions and definitions; readers who are unfamiliar with the following notions should refer to sources cited below for details. 
\begin{recollection}
	[Real topological K-theory]\label{rec:real_top_Ktheory} \cites{MR206940,MR2240234} 
	The space $ \Z \times BU $ has a $ C_2 $-action coming from complex conjugation on the unitary group $ U $, with $ C_2 $-fixed points $ \Z \times BO $. 
	Furthermore, there is a $ C_2 $-equivariant form of Bott periodicity $ \Z \times BU \simeq \Omega^{\rho}(\Z \times BU) $.
	Real K-theory $ KU_{\R} $ is the associated $ C_2 $-spectrum (under the equivalence in \cite[\S3]{GuillouMay}).

	By Appendix A (see discussion after Theorem A.5) of \cite{MR3214285}, $ KU_{\R} $ is an $ \E_\infty $ algebra in $ C_2 $-spectra. 
\end{recollection} 
\begin{recollection} [Poincaré $ \infty $-categories]
	There is an $ \infty $-category $ \Cat^p_{\infty} $ of \emph{Poincaré $ \infty $-categories} (\cite[Definitions 1.2.7-8]{CDHHLMNNSI}) whose objects are pairs $ (\cat, \Qoppa) $ consisting of a small stable $ \infty $-category and a quadratic functor $ \Qoppa: \cat^\op \to \Spectra $, and morphisms are given by duality-preserving exact functors. 
	Moreover, the $ \infty $-category $ \Cat^p_{\infty} $ has a symmetric monoidal structure \cite[Theorem 5.2.7(iii)]{CDHHLMNNSI} lifting the Lurie tensor product on small stable $ \infty $-categories. 
\end{recollection}
\begin{defn}[Real algebraic K-theory]\label{defn:real_alg_Ktheory} 
	Let $ A $ be a $ C_2 $-$ \E_\infty $-algebra. 
	We may associate to $ A $ the module with genuine involution $ (M= A^e, N = A^{\varphi C_2}, s_A : A^{\varphi C_2} \to A^{tC_2}) $ (Definition 3.2.2 of \emph{loc. cit}.
	To such a module with genuine involution there is an associated Poincaré $ \infty $-category $ \left(\Perf_{A^e},\Qoppa^{s_A}_{A^e} \right) $ (\cite[Construction 3.2.5]{CDHHLMNNSI}). 
	The \emph{real algebraic K-theory $ K_\R(A) $} of $ A $ is the real algebraic K-theory 
	\begin{equation*}
		K_\R(A) \simeq \mathrm{GW}^\mathrm{ghyp} \left(\Perf_{A^e},\Qoppa^{s_A}_{A^e} \right) \in \Spectra^{C_2} 	
	\end{equation*} 
	in the sense of \cite[Definition 4.5.1]{CDHHLMNNSII}. 
\end{defn}
\begin{prop}
\begin{enumerate}[label=(\arabic*)]
	\item \label{propitem:poincare_modules} The assignment $ A \mapsto \left(\Perf_{A^e},\Qoppa^{s_A}_{A^e} \right) $ of Definition \ref{defn:real_alg_Ktheory} (compare examples from \S3.2, in particular Example 3.2.11 of \cite{CDHHLMNNSI}) promotes to a symmetric monoidal functor
	\begin{equation*}
	\begin{split}
		C_2\E_\infty\Alg\left(\Spectra^{C_2}\right) \to \Cat^p_{\infty}
	\end{split}. 
	\end{equation*}

	\item \label{propitem:realKtheory_Einftyalg} The real algebraic K-theory of a $ C_2 $-$ \E_\infty $-algebra $ R $ canonically refines to an $ \E_\infty $-algebra in $ C_2 $-spectra. 
\end{enumerate}
\end{prop}
\begin{proof}
	To prove \ref{propitem:poincare_modules}, it suffices to observe that a morphism of $ C_2 $-$ \E_\infty $-algebras $ \phi \colon A \to B $ induces a canonical triple $ (\delta, \gamma, \sigma) $ in the sense of Corollary 3.3.2 of \cite{CDHHLMNNSI} (corresponding to a hermitian functor $ \left(\Perf_{A^e},\Qoppa^{s_A}_{A^e} \right) \to \left(\Perf_{B^e},\Qoppa^{s_B}_{B^e} \right) $covering the induction $ \phi_* \colon \Mod_{A^e} \to \Mod_{B^e} $).  
	Furthermore, the triple $ (\delta, \gamma, \sigma) $ automatically satisfies the criterion of Lemma 3.3.3 \& Definition 3.3.4 of \emph{loc.cit}., hence the associated hermitian functor is in fact Poincaré.  

	To prove \ref{propitem:realKtheory_Einftyalg}, it suffices to exhibit a composite functor $ C_2 \E_\infty \Alg\left(\Spectra^{C_2}\right) \to \Cat_\infty^p \to \Spectra^{C_2} $ which is lax symmetric monoidal. 
	The former functor is lax symmetric monoidal by \ref{propitem:poincare_modules}. 
	The latter functor is that of \cite[Definition 4.5.1]{CDHHLMNNSII}.  
	That it is lax symmetric monoidal will appear in \cite{CDHHLMNNSIV}. 
\end{proof}

A ring $ R $ is said to satisfy the \emph{homotopy limit problem} if its genuine symmetric real K-theory is a Borel $ C_2 $-spectrum \cites{MR711065}[\S3]{CDHHLMNNSIII}. 
\begin{cor}\label{cor:realKtheory_normed}
\begin{itemize}
	\item Real topological K-theory $ KU_{\R} $ admits a unique refinement to a $ C_2 $-$ \E_\infty $ ring spectrum. 
	\item If $ A $ satisfies the \emph{homotopy limit problem}, then $ K_{\R}(A) $ admits a unique refinement to a $ C_2 $-$\E_\infty $ ring spectrum. 
\end{itemize}	
\end{cor}
\begin{proof}
	By \cite{MR206940} (also see \cites[proof of Proposition 5.3.1]{MR2387923}[Corollary 7.6]{MR2240234}), $ KU_{\R} $ is Borel. 
	Both results follow from Proposition \ref{prop:Borel_are_normed}. 
\end{proof}

\subsection{A relative enhancement}
In this section we state a version of our main theorem \emph{relative} to an arbitrary base $ C_p $-$ \E_\infty $-algebra $ A $ (Example \ref{ex:param_calg}). 
In order to make sense of a $ C_p $-$ \E_\infty $-algebra over $ A $, we require a $ C_p $-symmetric monoidal structure on the category of $ A $-modules. 
That this is possible is suggested by the following
\begin{defn}\label{defn:relativenorm}
	Let $ A $ be a $ C_p $-$ \E_\infty $-ring in $ \Spectra^{C_p} $. 
	The \emph{($ A $-linear or relative) norm} is the functor 
	\begin{align*}
		\underline{N}_e^{C_p} \colon \Mod_{A^e} \xrightarrow{} \Mod_{A}\left(\Spectra^{C_p}\right) \\
		M \mapsto A \otimes_{{N}_e^{C_p}(A^e)}  N_e^{C_p} M
	\end{align*}
	Note that the reasoning of Lemma \ref{lemma:norm_cocontinuous_alg} applies to show that $ \underline{N}_e^{C_p} $ lifts to a colimit-preserving functor $ \E_\infty \Alg_{A^e} \to \E_\infty\Alg_{A}$.
\end{defn}
By Proposition \ref{prop:C2Einftymodules_are_C2sym_mon} (communicated by Jay Shah), we may regard the category of $ A $-modules as a $ C_p $-symmetric monoidal $ \infty $-category. 
\begin{defn}
	Let $ A $ be a $ C_p $-$ \E_\infty $-algebra (Example \ref{ex:param_calg}). 
	The $ \infty $-category of $ C_p $-$ \E_\infty $-$ A $-algebras is $ \Alg_{\underline{\Fin}_{C_p,*}}\left(\underline{\Fin}_{C_p,*},\left(\underline{\Mod_A}\right)^\otimes\right) =: C_p\E_\infty\Alg_A $ (Definition \ref{defn:param_alg}).
	In other words, it is the category of sections of $ \underline{\Mod}_A^{\otimes} \to \underline{\Fin}_{C_p,*} $ which take inert morphisms to inert morphisms. 
\end{defn}
There is moreover a relative notion of normed algebras over $ A $. 
\begin{defn}\label{defn:normedalgebras}
	Let $ A $ be a $ C_p $-$ \E_\infty $-ring. 
	We define the category $ N \E_\infty \Alg_A $ of \emph{normed $ \E_\infty $-$ A $-algebras} to be the limit of the diagram
	\begin{equation*}
	\begin{tikzcd}[column sep=small]
		& \Fun\left(\mathcal{O}_{C_p},\E_\infty\Alg_{A^e}\right) \ar[dd,"{\ev_{C_p},\ev_{C_p/C_p}}"]	& \\
		& & \Fun\left(\mathcal{O}_{C_p},\E_\infty\Alg_A\right) \ar[dd,"{\ev_{C_p}, \ev_{C_p/C_p}}"] \ar[lu,"{(-)^e}"'] \\
		&\E_\infty\Alg_{(A^e)^{\otimes p}} \times \E_\infty\Alg_{A^e} & \\
		\E_\infty\Alg_A  \ar[rr,"{N^{C_p}(-^e) \times \id}"'] \ar[ru] \ar[ruuu,"{m \circ (-^e)}", bend left=30] & &\E_\infty\Alg_{N^{C_p}(A^e)} \times \E_\infty\Alg_A \ar[lu,"{(-)^e \times (-)^e}"]
	\end{tikzcd}
	\end{equation*}
	where we have abbreviated $ \E_\infty\Alg_A = \E_\infty\Alg_A\left(\Spectra^{C_p}\right) $, $\E_\infty\Alg_{(A^e)^{\otimes p}} = \E_\infty\Alg_{(A^e)^{\otimes p}}\left(\Spectra^{BC_p \times BC_p}\right) $, etc.
\end{defn}
There is a relative version of the main results of this paper.
\begin{prop} \label{prop:rel_operadic_diagrammatic_comparison}
	Let $ A $ be a $ C_p $-$ \E_\infty $-ring. 
	There is a canonical forgetful functor
	\begin{equation*}
		\gamma_A \colon C_p \E_\infty \Alg_A \to N\E_\infty \Alg_A
	\end{equation*}
\end{prop}
\begin{proof}
	Proceeds as in proof of Corollary \ref{cor:operadic_diagrammatic_comparison}. 
\end{proof}
\begin{theorem}\label{thm:relative_mainthm}
	Let $ A $ be a $ C_p $-$ \E_\infty $-ring. 
	Then the canonical comparison functor $ \gamma_A $ of Proposition \ref{prop:rel_operadic_diagrammatic_comparison} is an equivalence. 
\end{theorem}
\begin{proof}
	Proceeds as in proof of Theorem \ref{thm:mainthm}. 
\end{proof}

\appendix

\section{Modules over normed equivariant algebras}
In this appendix, we show that the category of modules over a $ C_p $-$ \E_\infty $-ring naturally acquires a structure of a $ C_p $-symmetric monoidal $ \infty $-category in the sense of Nardin--Shah via a relative norm (cf. Definition \ref{defn:relativenorm}). 
The author would like to thank Jay Shah who communicated details of this construction. 

Fix $ \kappa $ a regular cardinal and let $ \mathcal{K} $ denote the collection of $ \kappa $-small simplicial sets. 
\begin{recollection}
	\cite[Notation 4.8.3.5.]{LurHA} 
	Write $ \Alg_{\E_1}(\Cat_\infty) $. 
	It has objects given by monoidal $ \infty $-categories which are compatible with $ \kappa $-indexed colimits and whose morphisms are monoidal functors $ F: \cat^\otimes \to \mathcal{D}^\otimes $ whose preserve $ \kappa $-indexed colimits. 
	Write $ U \colon \Alg_{\E_1}(\Cat_\infty) \to \Cat_\infty $ for the forgetful functor which forgets the monoidal structure. 

	\cite[Definition 4.8.3.7.]{LurHA}
	There is an $ \infty $-category $ \Cat^{\Alg}_\infty(\mathcal{K}) $. 
	Informally its objects are given by pairs $ (\cat^\otimes, A) $ where $ \cat $ is a monoidal $ \infty $-category and $ A $ is an algebra object of $ \cat $. 
	Morphisms from $ (\cat^\otimes,A) $ to $ (\mathcal{D}^\otimes,B) $ are given by monoidal functors $ F: \cat \to \mathcal{D} $ such that $ F(A) \simeq B $. 

	Similarly, there is an $ \infty $-category $ \Cat^{\Mod}_\infty(\mathcal{K}) $ whose objects are pairs $ (\cat^\otimes, \mathcal{M}) $ where $ \cat $ is a monoidal $ \infty $-category and $ \mathcal{M} $ is an $ \infty $-category left tensored over $ \cat $. 
	In particular, there is a forgetful functor 
	\begin{equation}\label{eq:forget_catmod_to_mod}
	\begin{split}
		\Upsilon \colon & \Cat^{\Mod}_\infty(\mathcal{K}) \to \Cat_\infty \\
			& (\cat^\otimes, \mathcal{M}) \mapsto \mathcal{M}
	\end{split} . 
	\end{equation}

	\cite[Construction 4.8.3.24]{LurHA} There is a functor $ \Theta $ making the following diagram commute: 
	\begin{equation}\label{diagram:algebras_categorification}
		\begin{tikzcd}
			\Cat^{\Alg}_\infty(\mathcal{K}) \ar[rr,"{\Theta:(\cat^\otimes, A) \mapsto (\cat^\otimes,\Mod_A(\cat))}"] \ar[rd,"{u_a}"'] & & \Cat^{\Mod}_\infty(\mathcal{K}) \ar[ld,"{u_m}"] \\
				& \Alg_{\E_1}(\Cat_\infty) &
		\end{tikzcd}
	\end{equation}
	 where the vertical arrows are the universal fibrations classifying families of $ \E_1 $-algebras in $ \cat $ and modules over said algebras in $ \cat $, respectively \cite[Remarks 4.8.3.8. \& 4.8.3.20.]{LurHA}.   
\end{recollection}
\begin{ntn}
	Hereafter, we drop $ \kappa, \mathcal{K} $ from notation. 
\end{ntn}

\begin{lemma}\label{lemma:normedalg_as_familyofalg}
	Let $ A $ be a $ C_p $-$ \E_\infty $-algebra object of $ \left(\underline{\Spectra}^{C_p}\right)^\otimes $, and recall the functor $ \zeta \colon \Span(\Fin_{C_p}) \to \Alg_{\E_1}(\Cat) $ of Recollection \ref{recollection:Cpcat_norms}. 
	Then $ A $ lifts to a cocartesian section $ \tilde{A} $ of $ \int \zeta $. 
\end{lemma}
\begin{proof}
	Let $ T \to C_p/H $ be an object of $ \underline{\Fin}_{C_p,*} $. 
	There is a natural functor $ \iota_T := - \times T \colon \Fin_* \to \underline{\Fin}_{C_p,*} $. 
	Moreover, $ \iota_{(-)} $ assembles to give the functor
	\begin{equation*}
	\begin{split}
		\iota \colon \Fin_* \times \underline{\Fin}_{C_p,*} & \to \underline{\Fin}_{C_p,*} \\
		(S, T \to C_p/H) & \mapsto S \times T \to C_p/H
	\end{split}. 
	\end{equation*}
	Consider the restriction $ \iota_T^*A $ of $ A $ along $ \iota_T $. 
	We may further precompose $ \iota_T^*A $ with the structure morphism $ \mathrm{Assoc}^\otimes \to \Fin_* $ where $ \mathrm{Assoc}^\otimes $ is the $ \E_1 $ operad of \cite[Definition 4.1.1.3]{LurHA}.
	Then by the characterization of inert morphisms of Theorem 2.3.3 of \cite{NS22}, $ \iota_T^*A $ and by definition of morphisms of operads (both parametrized and non-parametrized), $ \iota_T^*A $ defines an associative algebra object in $ \zeta(T) $. 
	Likewise $ \iota^*A $ defines a $ \underline{\Fin}_{C_p,*} $-family of associative algebra objects, hence the existence of $ \tilde{A} $ follows by the universal property characterizing $ \Cat^{\Alg}_\infty $.
\end{proof}
\begin{variant}\label{variant:normedalg_family_unit}
	Let $ \mathcal{K}_0 $ denote the full subcategory of the arrow category $ \Ar(\Fin_*) $ on those arrows given by the inclusion of the basepoint $ \{*\} \inj S $. 
	There is a variant functor
	\begin{equation*}
	\begin{split}
		\iota_\eta \colon \mathcal{K}_0  \times \underline{\Fin}_{C_p,*} & \to \underline{\Fin}_{C_p,*} \\
		(\{* \} \inj S, T) & \mapsto (\{* \} \inj S) \times T
	\end{split}
	\end{equation*}
	which defines a $ \Delta^1 \times \underline{\Fin}_{C_p,*} $-family of associative algebra objects. 
\end{variant}
Now consider the commutative diagram
\begin{equation*}
	\begin{tikzcd}
		\int \zeta \ar[d] \ar[r,"{u_a^*(\zeta)}"] & \Cat^{\Alg}_\infty \ar[r,"\Theta"] \ar[d,"{u_a}"] & \Cat^{\Mod}_\infty \ar[r,"\Upsilon"] \ar[d,"{u_m}"] & \Cat_\infty \\
		\underline{\Fin}_{C_p,*} \ar[r,"\zeta"] \ar[u,"\tilde{A}", bend left=30] & \Alg_{\E_1}(\Cat) \ar[r,equals] & \Alg_{\E_1}(\Cat) & 
	\end{tikzcd}
\end{equation*}
where $ \tilde{A} $ exists by Lemma \ref{lemma:normedalg_as_familyofalg} and the center square is (\ref{diagram:algebras_categorification}). 

\begin{defn}\label{defn:module_over_C2Einfty_alg}
	Let $ A $ be a $ C_p $-$ \E_\infty $-algebra in $ C_p $-spectra. 
	Recall the Grothendieck construction \cite[Theorem 2.2.1.2]{LurHTT}. 
	Let $ \underline{\Mod}_A^\otimes $ be the $ C_p $-symmetric monoidal $ \infty $-category classified by the morphism $ \Upsilon \circ \Theta \circ u_a^*(\zeta) \circ \tilde{A} $. 
	Let $ \underline{\Mod}_A $ denote the corresponding underlying $ C_p $-$ \infty $-category of $ \underline{\Mod}_A^\otimes $.  
\end{defn}
\begin{exs}
\begin{enumerate}[label=(\arabic*)]
	\item When $ A = \sphere^0 $, we recover $ \underline{\Mod}_A \simeq \underline{\Spectra}^{C_p} $. 
	\item Suppose $ A $ is a $ C_p $-$ \E_\infty $-ring. 
	Then $ \underline{\Mod}_A $ may be regarded as the $ \mathcal{O}_{C_p} $-diagram \emph{of stable $ \infty $-categories}
	\begin{equation*}
	\begin{tikzcd}
		\Mod_{A}\left(\Spectra^{C_p}\right) \ar[r,"{(-)^e}"] &	\Mod_{A^e}(\Spectra) \ar[loop below,"{\sigma_*}",out=-60, in=240,distance=15] 
	\end{tikzcd}  .
	\end{equation*}
	The morphism $ [C_p \to C_p/C_p] \to [C_p/C_p = C_p/C_p] $ in $ \underline{\Fin}_{C_p,*} $ classifies the relative norm $ \underline{N}^{C_p} $ of Definition \ref{defn:relativenorm}.
\end{enumerate}
\end{exs}
The previous discussion shows that 
\begin{prop}\label{prop:C2Einftymodules_are_C2sym_mon}
	Let $ A $ be a $ C_p $-$ \E_\infty $-algebra in $ C_p $-spectra. 
	Then the $ C_p $-$ \infty $-category of Definition \ref{defn:module_over_C2Einfty_alg} naturally acquires a $ C_p $-symmetric monoidal structure in the sense of \cite[Definition 2.2.1]{NS22}. 
\end{prop}
\begin{rmk}
	By the proof of Lemma \ref{lemma:normedalg_as_familyofalg}, each $ A_T $ for $ T \to C_p/H $ an object of $ \underline{\Fin}_{C_p,*} $ is in fact an $ \E_\infty $-algebra in $ \zeta(T) $. 
	Thus we write modules instead of left modules \cite[Corollary 4.5.1.6]{LurHA}. 
\end{rmk}
\begin{construction} [Parametrized base change]
	Since $ \underline{\Fin}_{C_p,*} $ is unital, the category of $ C_p $-$ \E_\infty $-algebras in $ C_p $-spectra has an initial object $ \underline{\mathbbm{1}} $ \cite[Definition 5.2.1 \& Theorem 5.2.11, resp.]{NS22} given fiberwise by the sphere spectrum. 
	As in Lemma \ref{lemma:normedalg_as_familyofalg}, $ \underline{\mathbbm{1}} $ lifts to a coCartesian section $ \tilde{\underline{\mathbbm{1}}} $ of $ \int \zeta $. 
	Suppose $ A $ is a $ C_p $-$ \E_\infty $-ring spectrum. 
	Variant \ref{variant:normedalg_family_unit} shows that the unit map $ \eta \colon \underline{\mathbbm{1}} \to A $ induces a natural transformation $ \tilde{\eta} \colon \tilde{\underline{\mathbbm{1}}} \to \tilde{A} $. 
	Under the Grothendieck construction, the unstraightening of $ \Upsilon \circ \Theta  \circ u^*_a(\zeta) (\eta) $ corresponds to a $ C_p $-functor of $ C_p $-$ \infty $-categories which we denote by 
	\begin{equation*}
	 	- \otimes_{\sphere^0} A \colon \underline{\Spectra}^{C_p} \to \underline{\Mod}_A .
	\end{equation*} 
\end{construction}
Categories of modules behave in the expected way. 
\begin{prop}
	Let $ A $ be a $ C_p $-$ \E_\infty $-algebra in $ C_p $-spectra. 
	Then the $ C_p $-functor $ - \otimes_{\sphere^0} A \colon \underline{\Spectra}^{C_p} \to \underline{\Mod}_A $ participates in a $ C_p $-adjunction which is fiberwise monadic. 
\end{prop}
\begin{proof}
	Notice that essentially by definition, the $ C_p $-functor $ - \otimes_{\sphere^0} A $ preserves coCartesian arrows. 
	By (the dual to) \cite[Proposition 7.3.2.6]{LurHA}, it suffices to check that $ - \otimes_{\sphere^0} A $  admits a fiberwise right adjoint, which is classical. 
\end{proof}
\begin{rmk}
	The strategy outlined here generalizes straightforwardly to endow the $ G $-$ \infty $-category of modules over a normed $ G $-$ \E_\infty $-ring spectrum with the structure of a $ G $-symmetric monoidal structure for any finite group $ G $. 
\end{rmk}
\begin{rmk}
	One expects a equivariant form of the Tannaka reconstruction theorem \cite[Propositions 7.1.2.6-7]{LurHA} by which a $ G $-$ \E_\infty $-ring $ A $ can be recovered from its category of modules endowed with its $ G $-symmetric monoidal structure. 
\end{rmk}

\renewcommand*{\bibfont}{\small}
{ \printbibliography}
\addcontentsline{toc}{section*}{References} 

\end{document}